\documentclass[12pt, reqno]{amsart}
\usepackage{amsmath, amsthm, amscd, amsfonts, amssymb, graphicx, color, mathrsfs}
\usepackage[bookmarksnumbered, colorlinks, plainpages]{hyperref}
\usepackage{tikz}
\usetikzlibrary{arrows.meta}
\hypersetup{colorlinks=true,linkcolor=red, anchorcolor=green, citecolor=cyan, urlcolor=red, filecolor=magenta, pdftoolbar=true}
\usepackage{hyperref}
\newtheorem{theorem}{Theorem}[section]
\newtheorem{lemma}[theorem]{Lemma}

\theoremstyle{definition}

\theoremstyle{remark}
\newtheorem{remark}[theorem]{Remark}
\numberwithin{equation}{section}

\begin{document}

\setcounter{page}{1}

\title[Integral representations and asymptotic behaviors \dots ]{Integral representations and asymptotic behaviors of the Multivariate Mittag-Leffler function}

\author[D. Shamuratov]{Damir Shamuratov}

\thanks{}

\subjclass[2010]{}


\begin{abstract}
 In this paper, a multivariate Mittag--Leffler-type function arising in the theory of fractional differential equations with several fractional parameters is investigated. New Hankel contour integral representations are derived for the three-variable Mittag--Leffler function, and complete asymptotic expansions are established in different sectors of the complex plane. The proposed approach is based on the classical Hankel integral representation of the reciprocal Gamma function together with suitable contour transformations. Furthermore, the corresponding integral representations and asymptotic expansions are established for the multivariate Mittag--Leffler function with an arbitrary number of variables. The obtained formulas extend several known results for one- and two-variable Mittag--Leffler functions and provide useful analytical tools for the qualitative analysis of fractional differential equations involving multiple fractional derivatives.
\end{abstract} \maketitle

\section{Introduction}

        Fractional differential equations have attracted considerable attention during the last few decades due to their ability to provide more accurate mathematical models for numerous real-world phenomena than classical integer-order differential equations. Such equations naturally arise in various areas of science and engineering, including anomalous diffusion, viscoelasticity, heat conduction, porous media, control theory, signal processing, electrical circuits, biology, finance, and many other fields. The presence of memory and hereditary effects makes fractional differential equations particularly suitable for describing complex dynamical processes.

One of the most important mathematical tools in the theory of fractional differential equations is the Mittag--Leffler function. Similar to the role played by the exponential function in classical differential equations, Mittag--Leffler functions naturally appear in explicit solution formulas of fractional ordinary and partial differential equations. Consequently, their analytical properties are of fundamental importance for investigating qualitative properties of solutions, such as existence, uniqueness, stability, decay estimates, long-time asymptotic behavior, and inverse problems.

For this reason, various generalizations of the classical Mittag--Leffler function have been introduced and extensively investigated in the literature \cite{Gorenflo},\cite{Kilbas},\cite{Miller},\cite{Podlubny}. Their integral representations, asymptotic expansions, recurrence relations, differentiation formulas, and other analytical properties have been studied by many authors. These investigations have significantly contributed to the development of fractional calculus and its applications.

In \cite{Ogorodnikov1}, the authors considered the following fractional differential equations:
\begin{equation}
\ddot{u}(t)+pD_{0t}^{\beta}\dot{u}(t)+qD_{0t}^{2\beta}u(t)=f(t),
\label{eq:Ogorotnikov1}
\end{equation}
and
\begin{equation}
\ddot{u}(t)+pD_{0t}^{1+\beta}u(t)+qD_{0t}^{2\beta}u(t)=f(t),
\label{eq:Ogorotnikov2}
\end{equation}
where \(p,q\in\mathbb{R}\), \(\beta\in(0,\frac12)\), and \(f(t)\in L(0,T)\) is a given function.
 In the process of deriving explicit
solution formulas, the composition of two Mittag--Leffler integral operators was
evaluated. This led to the introduction of the integral operator
\[
E^{\mu;\alpha,\beta}_{0t;\lambda_1,\lambda_2}f
=
\int_0^t
(t-\tau)^{\mu-1}
E_{\alpha,\beta}\!\left(
\lambda_1(t-\tau)^\alpha,
\lambda_2(t-\tau)^\beta;\mu
\right)
f(\tau)\,d\tau,
\]
whose kernel is the two-variable Mittag--Leffler function
\[
E_{\alpha,\beta}(x,y;\mu)
=
\sum_{n,m=0}^{\infty}
\frac{x^n y^m}
{\Gamma(n\alpha+m\beta+\mu)}.
\]

Equations \eqref{eq:Ogorotnikov1} and \eqref{eq:Ogorotnikov2} can be considered as model differential equations of fractional oscillators (see \cite{Machado}). 
A general linear fractional oscillator with constant coefficients is governed by the equation (see \cite{Ogorodnikov1}, \cite{Ogorodnikov2})
\begin{equation}\label{eq:ogorotnikov3}
\sum_{k=0}^{n}a_kD_{0t}^{\alpha_k}\dot{u}(t)
+
\sum_{k=0}^{m}b_kD_{0t}^{\beta_k}u(t)
=
f(t),
\end{equation}
where \(a_k,b_k\in\mathbb{R}\), \(\alpha_k\in(0,1]\), \(\beta_k\in[0,2)\), and \(f(t)\) is a prescribed external force. 

The solutions of such equation involve multivariate Mittag-Leffler-type functions similar to (see \cite{Gorenflo}, p. 195) 

\begin{equation}\label{multi variable ML}
E_{\alpha_1,\dots,\alpha_n}(z_1,\dots,z_n;\mu)
=
\sum_{k_1,k_2,\dots,k_n=0}^{\infty}
\frac{\prod_{j=1}^{n}z_j^{k_j}}
{\Gamma\!\left(\mu+\sum_{j=1}^{n}\alpha_jk_j\right)},
\end{equation}
where $\alpha_j>0$ (\(j=1,\ldots,n\)), and \(\mu, z_1,z_2,\dots,z_n\in\mathbb{C}\).

 Multi-term fractional differential equations and \eqref{eq:ogorotnikov3} fractional differential equations and related fractional evolution equations continue to attract considerable attention due to their ability to model complex dynamical processes involving multiple memory effects. In the analysis of such equations, multivariable, multinomial, and multi-index Mittag--Leffler functions play a fundamental role, as they naturally arise in explicit solution representations of fractional differential equations with several fractional orders (for more information, see \cite{Ashurov},\cite{Kiryakova},\cite{Yamamoto},\cite{Luchko},\cite{Ogorodnikov2},\cite{Umarov}). Consequently, investigating the analytical properties of these functions, including their integral representations and asymptotic behavior, is of significant importance for both the theory and applications of  fractional differential equations.

The main purpose of the present paper is to investigate a multivariate Mittag--Leffler-type function of $n$ variables. We derive new Hankel integral representations and establish complete asymptotic expansions in different sectors of the complex domain. The obtained results extend several previously known asymptotic formulas for one- and two-variable Mittag--Leffler functions and provide useful analytical tools for the study of fractional differential equations involving several fractional parameters.

The remainder of this paper is organized as follows. In Section \ref{sec2}, we present the necessary preliminaries and recall several known results on two-variable Mittag--Leffler functions that will be used throughout the paper. In Section \ref{sec3}, we derive new Hankel integral representations for the three-variable Mittag--Leffler function. Section \ref{sec4} is devoted to the derivation of complete asymptotic expansions in different sectors of the complex plane. Finally, Section \ref{sec5} extends the obtained results to the general multivariate Mittag--Leffler function and presents the corresponding asymptotic formulas.

\section{Preliminaries}\label{sec2}

The derivation of the integral representations and asymptotic expansions presented in this paper is based on the Hankel contour method. In particular, the Hankel contour allows us to employ the classical integral representation of the reciprocal Gamma function and serves as the main analytical tool in the subsequent analysis. For the reader's convenience, we first recall its definition and notation.

The Hankel path considered in Lemmas \ref{l2.1}, \ref{l2.2}, \ref{l2.3},
\ref{l3.1}, \ref{l3.2}, \ref{l5.1}, \ref{l5.2} and Theorems
\ref{thm4.1}, \ref{thm5.3} is denoted by
$
\omega(\varepsilon;\eta),
$
and is oriented in the direction of non-decreasing $\arg s$.
It consists of the following parts:
\begin{itemize}
\item[(1)] the two rays
$
S_\eta = \{\arg s = \eta,\ |s| \ge \varepsilon\},
\qquad
S_{-\eta} = \{\arg s = -\eta,\ |s| \ge \varepsilon\};
$

\item[(2)] the circular arc
$
C_\eta(0;\varepsilon)
=
\{|s|=\varepsilon,\ -\eta \le \arg s \le \eta\}.
$
\end{itemize}

If $0<\eta<\pi$, then the contour $\omega(\varepsilon;\eta)$ divides the complex $s$-plane into two unbounded regions, namely
$
\Omega^{(-)}(\varepsilon;\eta)
$ to the left of $ \omega(\varepsilon;\eta),
$
and
$
\Omega^{(+)}(\varepsilon;\eta)$
 to the right of $\omega(\varepsilon;\eta)
$ (see Figure 1).

\begin{figure}[ht]
\centering
\begin{tikzpicture}[scale=3,>=stealth]

\def\r{0.55}
\def\a{140}

\fill[gray!25]
(-1.05,0.78)
--
({\r*cos(\a)},{\r*sin(\a)})
arc (\a:-140:\r)
--
(-1.05,-0.78)
-- cycle;

\draw[->] (-1.05,0) -- (1.25,0);
\draw[->] (0,-0.95) -- (0,1.05);

\draw[line width=1pt]
({\r*cos(\a)},{\r*sin(\a)})
arc (\a:-140:\r);

\draw[->,line width=1pt]
({\r*cos(-55)},{\r*sin(-55)})
arc (-55:-35:\r);

\draw[line width=1pt]
({\r*cos(\a)},{\r*sin(\a)})
-- (-0.98,0.74);

\draw[line width=1pt]
({\r*cos(220)},{\r*sin(220)})
-- (-0.98,-0.74);

\draw[dashed]
(0,0)--({0.42*cos(\a)},{0.42*sin(\a)});

\draw[dashed]
(0,0)--({0.42*cos(-\a)},{0.42*sin(-\a)});

\draw (0.18,0) arc (0:\a:0.18);
\draw (0.15,0) arc (0:-\a:0.15);

\node at (0.12,0.23) {$\eta$};
\node at (0.13,-0.24) {$-\eta$};

\node at (-0.68,0.67) {$S_{\eta}$};
\node at (-0.63,0.25) {$\Omega^{(-)}(\varepsilon;\eta)$};

\node at (-0.70,-0.70) {$S_{-\eta}$};

\node at (0.62,0.82)
{$\Omega^{(+)}(\varepsilon;\eta)$};

\node at (0.60,0.32) {$C_{\eta}$};

\node at (0.58,-0.50)
{$\omega(\varepsilon;\eta)$};

\draw (0,0)--({\r*cos(35)},{\r*sin(35)});
\node at (0.34,0.18) {$\varepsilon$};

\end{tikzpicture}
\caption{The Hankel contour $\omega(\varepsilon;\eta)$.}
\end{figure}
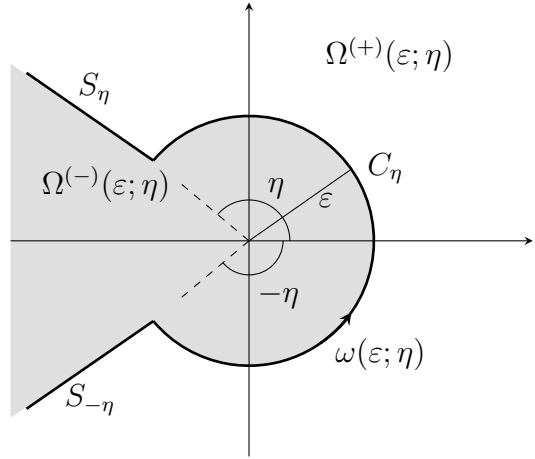

If $\eta=\pi$, then the contour consists of the circle
$
\{|s|=\varepsilon\}
$
and the ray $-\infty < s \le -\varepsilon$ ($\varepsilon>0$), which is a two-way path along the real axis.
More precisely, this keyhole (Hankel) contour is a path starting from $-\infty$, moving along the real axis to $-\varepsilon$, then going counterclockwise around a circle of radius $\varepsilon$ centered at $0$, returning to $-\varepsilon$, and finally continuing back to $-\infty$ along the real axis.

Since the proofs of the integral representations and asymptotic expansions for the multivariate Mittag--Leffler function are based on the same ideas and arguments as those for the one-, two-, and three-variable cases, we first derive these results for the three-variable Mittag--Leffler function. This approach allows us to present the essential techniques without unnecessary repetition, while the extension to the general multivariable case follows in a straightforward manner.
The three-variable Mittag–Leffler function is defined by
$$E_{\alpha,\beta,\gamma}(x,y,z;\mu)
=
\sum_{n,m,p=0}^{\infty}
\frac{x^n y^m z^p}
{\Gamma(n\alpha+m\beta+p\gamma+\mu)},$$
where \(\alpha,\beta,\gamma>0\), and \(\mu,x,y,z\in\mathbb{C}\).

We recall the following results established by C. Lavault. 
\begin{lemma}\cite{Lavault}\label{l2.1}
    Let $0 < \alpha, \beta < 2$ and $\alpha\beta < 2$. Let $\mu$ be any complex number and let $\eta$ satisfy the condition
\begin{equation}\label{2.1}
\frac{\pi\alpha\beta}{2}
<
\eta
\leq
\min\!\bigl(\pi,\pi\alpha\beta\bigr).
\end{equation}

If
$
x \in \Omega^{(-)}(\varepsilon_\alpha;\eta_\alpha)
\quad\text{and}\quad
y \in \Omega^{(-)}(\varepsilon_\beta;\eta_\beta),
$
where
$
\varepsilon_\alpha:=\varepsilon^{1/\beta},
\varepsilon_\beta:=\varepsilon^{1/\alpha},
\eta_\alpha:=\cfrac{\eta}{\beta},
\eta_\beta:=\cfrac{\eta}{\alpha},
$
then the following Hankel integral representation holds:
\[
E_{\alpha,\beta}(x,y;\mu)
=
\frac{1}{2\pi i}\,
\frac{1}{\alpha\beta}
\int_{\omega(\varepsilon;\eta)}
e^{s^{1/(\alpha\beta)}}
s^{\frac{\alpha+\beta+1-\mu}{\alpha\beta}-1}
\frac{1}
{\left(s^{1/\alpha}-y\right)
 \left(s^{1/\beta}-x\right)}
\,ds .
\]
\end{lemma}

\begin{lemma}\cite{Lavault}\label{l2.2}
    Let $0 < \alpha, \beta < 2$ and $\alpha\beta < 2$. Let $\mu$ be any complex number and let $\eta$ satisfy the condition \eqref{2.1}.
If
$
x \in \Omega^{(+)}(\varepsilon_\alpha;\eta_\alpha)
\quad\text{and}\quad
y \in \Omega^{(+)}(\varepsilon_\beta;\eta_\beta),
$
where
$
\varepsilon_\alpha:=\varepsilon^{1/\beta},
\varepsilon_\beta:=\varepsilon^{1/\alpha},
\eta_\alpha:=\cfrac{\eta}{\beta},
\eta_\beta:=\cfrac{\eta}{\alpha},
$
then the following Hankel integral representation holds:
\[
E_{\alpha,\beta}(x,y;\mu)
=
\frac{1}{\alpha}
\frac{e^{x^{1/\alpha}}x^{\frac{1+\beta-\mu}{\alpha}}}
{x^{\beta/\alpha}-y}
+
\frac{1}{\beta}
\frac{e^{y^{1/\beta}}y^{\frac{1+\alpha-\mu}{\beta}}}
{y^{\alpha/\beta}-x}
\]
\[+\frac{1}{2\pi i}\,
\frac{1}{\alpha\beta}
\int_{\omega(\varepsilon;\eta)}
e^{s^{1/(\alpha\beta)}}
s^{\frac{\alpha+\beta+1-\mu}{\alpha\beta}-1}
\frac{1}
{\left(s^{1/\alpha}-y\right)
 \left(s^{1/\beta}-x\right)}
\,ds .
\]
\end{lemma}

\begin{lemma}\cite{Lavault}\label{l2.3}
    Let $0<\alpha,\beta<2$ and $\alpha\beta<2$. Let $\mu$ be any complex number and
$\tau$ be any real number satisfying the inequalities
\begin{equation}
\frac{\pi\alpha\beta}{2}
<
\tau
<
\min\!\bigl(\pi,\pi\alpha\beta\bigr).
\end{equation}

Then, for all integers $r_{\alpha},r_{\beta}\geq 1$, the function
$E_{\alpha,\beta}(x,y;\mu)$ verifies the following asymptotic formulas
drawn from its integral representations whenever $|x|\to\infty$ and
$|y|\to\infty$.
\begin{itemize}
    \item 
If  $|\arg x|\le \cfrac{\tau}{\beta}
$ and 
$|\arg y|\le \cfrac{\tau}{\alpha}$,
then
\[
\begin{aligned}
E_{\alpha,\beta}(x,y;\mu)
={}&
\frac{1}{\alpha}
\frac{e^{x^{1/\alpha}}
x^{\frac{1+\beta-\mu}{\alpha}}}
{x^{\beta/\alpha}-y}
+
\frac{1}{\beta}
\frac{e^{y^{1/\beta}}
y^{\frac{1+\alpha-\mu}{\beta}}}
{y^{\alpha/\beta}-x}
+
\sum_{n=1}^{r_{\beta}}
\sum_{m=1}^{r_{\alpha}}
\frac{x^{-n}y^{-m}}
{\Gamma(\mu-n\alpha-m\beta)}
\\[2mm]
&\quad
+\,o\!\left(|xy|^{-1}|x|^{-r_{\alpha}}\right)
+o\!\left(|xy|^{-1}|y|^{-r_{\beta}}\right).
\end{aligned}
\]

\item If 
$\cfrac{\tau}{\beta}\leq|\arg x|\le \pi$
and $\cfrac{\tau}{\alpha}\leq|\arg y|\le \pi,$ then
\[
E_{\alpha,\beta}(x,y;\mu)
=
\sum_{n=1}^{r_{\beta}}
\sum_{m=1}^{r_{\alpha}}
\frac{x^{-n}y^{-m}}
{\Gamma(\mu-n\alpha-m\beta)}
+
o\!\left(|xy|^{-1}|x|^{-r_{\alpha}}\right)
+
o\!\left(|xy|^{-1}|y|^{-r_{\beta}}\right).
\]
\end{itemize}

\end{lemma}

\section{Integral representations of three-variable Mittag-Leffler function}\label{sec3}

\begin{lemma}\label{l3.1}
     Let $0 < \alpha, \beta,\gamma < 2$ and $\alpha\beta\gamma < 2$. Let $\mu$ be any complex number and let $\eta$ satisfy the condition
\begin{equation}\label{3.1}
\frac{\pi\alpha\beta\gamma}{2}
<
\eta
\leq
\min\!\bigl(\pi,\pi\alpha\beta\gamma\bigr).
\end{equation}

If
$
x \in \Omega^{(-)}({\varepsilon}_\alpha;\eta_\alpha)
$,\,\, 
$y \in \Omega^{(-)}({\varepsilon}_\beta;\eta_\beta),
$ and $z \in \Omega^{(-)}({\varepsilon}_\gamma;\eta_\gamma)
$,
where
$
{\varepsilon}_\alpha:={\varepsilon}^{\frac{1}{\beta\gamma}},\,
{\varepsilon}_\beta:={\varepsilon}^{\frac{1}{\alpha\gamma}}, \,{\varepsilon}_\gamma:={\varepsilon}^{\frac{1}{\alpha\beta}},\,
\eta_\alpha:=\cfrac{{\eta}}{\beta\gamma},\,
\eta_\beta:=\cfrac{{\eta}}{\alpha\gamma},\,
\eta_\gamma:=\cfrac{{\eta}}{\alpha\beta},
$
then the following Hankel integral representation holds:
\begin{equation}\label{three-variable integral representation}
E_{\alpha,\beta,\gamma}(x,y,z;\mu)
=
\frac{1}{2\pi i}\,
\frac{1}{\alpha\beta\gamma}
\int_{\omega(\varepsilon;\eta)}
\frac{
e^{s^{1/(\alpha\beta\gamma)}}
\,s^{\frac{\alpha+\beta+\gamma+1-\mu}{\alpha\beta\gamma}-1}
}
{
(s^{1/\beta\gamma}-x)(s^{1/\alpha\gamma}-y)(s^{1/\alpha\beta}-z)
}
\,ds.
\end{equation}
\end{lemma}\begin{proof}
We rewrite the three-variable Mittag--Leffler function in the following form:
\begin{equation}\label{1.1}
\begin{aligned}
&E_{\alpha,\beta,\gamma}(x,y,z;\mu)
=
\sum_{n=0}^{\infty}
\sum_{m=0}^{\infty}\sum_{p=0}^{\infty}
\frac{x^n y^m z^p}
{\Gamma(n\alpha+m\beta+p\gamma+\mu)}\\
&=
\sum_{n=0}^{\infty}
\sum_{m=0}^{\infty}x^ny^m
\sum_{p=0}^{\infty}
\frac{z^p }
{\Gamma\bigl(p\gamma+(n\alpha+m\beta+\mu)\bigr)}=
\sum_{n=0}^{\infty}
\sum_{m=0}^{\infty}x^ny^m
E_{\gamma}(z;n\alpha+m\beta+\mu),
\end{aligned}\end{equation}
 where 
 $$E_{\gamma}(z;\mu)=\sum_{p=0}^{\infty}\frac{z^p}{\Gamma(p\gamma+\mu)},\quad \gamma>0,\quad\mu\in \mathbb{C},$$
 two-parametric Mittag-Leffler function (see \cite{Podlubny}, p. 17).
 
Under the assumptions of Lemma \ref{l3.1}, one can employ the known integral representation of
$E_{\gamma}(z;\alpha n+\beta m+\mu)$ , taking the above
$\varepsilon_{\gamma}$ and $\eta_{\gamma}$ as the parameters defining the Hankel contour.
This choice is admissible in view of inequalities \eqref{3.1}. Consequently, for
$
z\in\Omega^{(-)}(\varepsilon_{\gamma};\eta_{\gamma}),
$
and provided that
$
\eta_{\gamma}=\cfrac{\eta}{\alpha\beta},
$
the following representation follows from the integral representation of
$E_{\gamma}(z;\alpha n+\beta m+\mu)$ (see \cite{Podlubny}, p. 30):
\begin{equation}\label{two parametic integral representation}
E_{\gamma}(z;n\alpha+m\beta+\mu)
=
\frac{1}{2\pi i}\,
\frac{1}{\gamma}
\int_{\omega(\varepsilon_\gamma;\eta_\gamma)}
\frac{
e^{s^{1/\gamma}}
\,s^{\frac{1-\alpha n-\beta m-\mu}{\gamma}}
}
{
(s-z)
}
\,ds.
\end{equation}

Let $|x|<\varepsilon_{\alpha}$ and $|y|<\varepsilon_{\beta}$. Taking into account the fact that $\varepsilon_\alpha=\varepsilon^{1/\beta\gamma}=\left(\varepsilon_{\gamma}^{\alpha\beta}\right)^{1/\beta\gamma}=\varepsilon_{\gamma}^{\alpha/\gamma}$ and $\varepsilon_\beta=\varepsilon^{1/\alpha\gamma}=\left(\varepsilon_{\gamma}^{\alpha\beta}\right)^{1/\alpha\gamma}=\varepsilon_{\gamma}^{\beta/\gamma}$ yields next the inequalities
\begin{equation}\label{3.4}
\sup_{s\in\omega(\varepsilon_{\gamma};\eta_{\gamma})}
\left|xs^{-\alpha/\gamma}\right|<1,
\end{equation}
 \begin{equation}\label{3.5}
\sup_{s\in\omega(\varepsilon_{\gamma};\eta_{\gamma})}
\left|ys^{-\beta/\gamma}\right|<1.
\end{equation}

Using \eqref{two parametic integral representation}, we rewrite \eqref{1.1} in the following form:
\begin{equation}\label{1.2}
\begin{aligned}
&E_{\alpha,\beta,\gamma}(x,y,z;\mu)
=\frac{1}{2\pi i}\,
\frac{1}{\gamma}
\int_{\omega(\varepsilon_\gamma;\eta_\gamma)}
\frac{
e^{s^{1/\gamma}}
\,s^{\frac{1-\mu}{\gamma}}
}
{
(s-z)
}\left(\sum_{n=0}^{\infty}\left(xs^{-\alpha/\gamma}\right)^n\sum_{m=0}^{\infty}\left(ys^{-\beta/\gamma}\right)^m\right)
\,ds.
\end{aligned}\end{equation}

Simplifying and using \eqref{3.4}, \eqref{3.5} inequalities, we rewrite \eqref{1.2} in the following form:

\begin{equation}\label{1.3}
\begin{aligned}
E_{\alpha,\beta,\gamma}(x,y,z;\mu)
&=
\frac{1}{2\pi i}\,
\frac{1}{\gamma}
\int_{\omega(\varepsilon_\gamma;\eta_\gamma)}
\frac{
e^{s^{1/\gamma}}
\,s^{\frac{\alpha+\beta+1-\mu}{\gamma}}
}
{
(s^{\alpha/\gamma}-x)(s^{\beta/\gamma}-y)(s-z)
}
\,ds.
\end{aligned}\end{equation}
We introduce the notation
$s=\zeta^{\frac{1}{\alpha\beta}}$, $ds=\cfrac{1}{\alpha\beta}\zeta^{\frac{1-\alpha\beta}{\alpha\beta}}d\zeta$, we get
\begin{equation}\label{1.4}
E_{\alpha,\beta,\gamma}(x,y,z;\mu)
=
\frac{1}{2\pi i}\,
\frac{1}{\alpha\beta\gamma}
\int_{\omega(\varepsilon;\eta)}
\frac{
e^{\zeta^{1/(\alpha\beta\gamma)}}
\,\zeta^{\frac{\alpha+\beta+\gamma+1-\mu}{\alpha\beta\gamma}-1}
}
{
(\zeta^{1/\gamma\beta}-x)(\zeta^{1/\gamma\alpha}-y)(\zeta^{1/\alpha\beta}-z)
}
\,d\zeta.\end{equation}

The resulting integral is absolutely convergent and defines an analytic function of
\[
x\in\Omega^{(-)}(\varepsilon_\alpha;\eta_\alpha),\qquad
y\in\Omega^{(-)}(\varepsilon_\beta;\eta_\beta),\qquad
z\in\Omega^{(-)}(\varepsilon_\gamma;\eta_\gamma).
\]
Moreover,
$
|x|<\varepsilon_\alpha,\,
|y|<\varepsilon_\beta,\,
|z|<\varepsilon_\gamma
$
imply that
\[
x\in\Omega^{(-)}(\varepsilon_\alpha;\eta_\alpha),\qquad
y\in\Omega^{(-)}(\varepsilon_\beta;\eta_\beta),\qquad
z\in\Omega^{(-)}(\varepsilon_\gamma;\eta_\gamma).
\]
Therefore, by the principle of analytic continuation, the obtained integral representation remains valid throughout
\[
\Omega^{(-)}(\varepsilon_\alpha;\eta_\alpha)
\times
\Omega^{(-)}(\varepsilon_\beta;\eta_\beta)
\times
\Omega^{(-)}(\varepsilon_\gamma;\eta_\gamma),
\]
which completes the proof. \qed
\end{proof}

\begin{lemma}\label{l3.2}
     Let $0 < \alpha, \beta,\gamma < 2$ and $\alpha\beta\gamma < 2$. Let $\mu$ be any complex number and let $\eta$ satisfy the condition \eqref{3.1}.
If
$
x \in \Omega^{(+)}({\varepsilon}_\alpha;\eta_\alpha)
$,\,\, 
$y \in \Omega^{(+)}({\varepsilon}_\beta;\eta_\beta),
$ and $z \in \Omega^{(+)}({\varepsilon}_\gamma;\eta_\gamma)
$,
where
$
{\varepsilon}_\alpha:={\varepsilon}^{\frac{1}{\beta\gamma}},\,
{\varepsilon}_\beta:={\varepsilon}^{\frac{1}{\alpha\gamma}}, \,{\varepsilon}_\gamma:={\varepsilon}^{\frac{1}{\alpha\beta}},\,
\eta_\alpha:=\cfrac{{\eta}}{\beta\gamma},\,
\eta_\beta:=\cfrac{{\eta}}{\alpha\gamma},\,
\eta_\gamma:=\cfrac{{\eta}}{\alpha\beta},
$
then the following Hankel integral representation holds:
\begin{equation}\label{3.9}\begin{aligned}
&E_{\alpha,\beta,\gamma}(x,y,z;\mu)\\
&=\frac{1}{\alpha}\frac{
e^{x^{1/\alpha}}
\,x^{\frac{\beta+\gamma+1-\mu}{\alpha}}}
{(x^{\beta/\alpha}-y)(x^{\gamma/\alpha}-z)}+\frac{1}{\beta}\frac{
e^{y^{1/\beta}}
\,y^{\frac{\alpha+\gamma+1-\mu}{\beta}}}
{(y^{\alpha/\beta}-x)(y^{\gamma/\beta}-z)}+\frac{1}{\gamma}
\frac{
e^{z^{1/\gamma}}
\,z^{\frac{\alpha+\beta+1-\mu}{\gamma}}}
{(z^{\alpha/\gamma}-x)(z^{\beta/\gamma}-y)}\\
&+
\frac{1}{2\pi i}\,
\frac{1}{\alpha\beta\gamma}
\int_{\omega(\varepsilon;\eta)}
\frac{
e^{s^{1/(\alpha\beta\gamma)}}
\,s^{\frac{\alpha+\beta+\gamma+1-\mu}{\alpha\beta\gamma}-1}
}
{
(s^{1/\beta\gamma}-x)(s^{1/\alpha\gamma}-y)(s^{1/\alpha\beta}-z)
}
\,ds.\end{aligned}
\end{equation}
\end{lemma}
\begin{proof}
   First, suppose that only the point $z$ lies to the right of the Hankel contour
$\omega(\varepsilon_\gamma;\eta_\gamma)$, namely,
$
z\in\Omega^{(+)}(\varepsilon_\gamma;\eta_\gamma).
$
Choose an arbitrary number $\widetilde{\varepsilon}_{\gamma}$ satisfying
$\widetilde{\varepsilon}_{\gamma}>|z|$. Then
$
z\in\Omega^{(-)}(\widetilde{\varepsilon}_{\gamma};\eta_\gamma).
$
Furthermore, setting
$
\widetilde{\varepsilon}_{\alpha}
=\widetilde{\varepsilon}_{\gamma}^{\frac{\alpha}{\gamma}},
\,\,
\widetilde{\varepsilon}_{\beta}
=\widetilde{\varepsilon}_{\gamma}^{\frac{\beta}{\gamma}},
$
it follows that
$
x\in\Omega^{(-)}(\widetilde{\varepsilon}_{\alpha};\eta_\alpha),
\,\,
y\in\Omega^{(-)}(\widetilde{\varepsilon}_{\beta};\eta_\beta).
$
Hence, applying the integral representation \eqref{1.3}, we arrive at
\begin{equation}\label{3.10}E_{\alpha,\beta,\gamma}(x,y,z;\mu)
=
\frac{1}{2\pi i}\,
\frac{1}{\gamma}
\int_{\omega(\widetilde{\varepsilon}_{\gamma};\eta_\gamma)}
\frac{
e^{s^{1/\gamma}}
\,s^{\frac{\alpha+\beta+1-\mu}{\gamma}}
}
{
(s^{\alpha/\gamma}-x)(s^{\beta/\gamma}-y)(s-z)
}
\,ds.\end{equation}

Moreover, if
$
\varepsilon_\gamma<|z|<\widetilde{\varepsilon}_{\gamma},
$
then
$
|\arg z|<\eta_\gamma,
$
and, by Cauchy's theorem,
\begin{equation}\label{3.11}\begin{aligned}
E_{\alpha,\beta,\gamma}(x,y,z;\mu)
&=
\frac{1}{2\pi i}\,
\frac{1}{\gamma}
\int_{\omega(\widetilde{\varepsilon}_{\gamma};\eta_\gamma)
-
\omega(\varepsilon_\gamma;\eta_\gamma)}
\frac{
e^{s^{1/\gamma}}
\,s^{\frac{\alpha+\beta+1-\mu}{\gamma}}
}
{
(s^{\alpha/\gamma}-x)(s^{\beta/\gamma}-y)(s-z)
}ds\\
&=\frac{1}{\gamma}
\frac{
e^{z^{1/\gamma}}
\,z^{\frac{\alpha+\beta+1-\mu}{\gamma}}}
{(z^{\alpha/\gamma}-x)(z^{\beta/\gamma}-y)}\end{aligned}
\end{equation}

Therefore, from \eqref{3.10} and \eqref{3.11}, we obtain the following integral representation 
\begin{equation}\label{3.12}\begin{aligned}
E_{\alpha,\beta,\gamma}(x,y,z;\mu)
&=\frac{1}{\gamma}\frac{
e^{z^{1/\gamma}}
\,z^{\frac{\alpha+\beta+1-\mu}{\gamma}}}
{(z^{\alpha/\gamma}-x)(z^{\beta/\gamma}-y)}\\
&+
\frac{1}{2\pi i}\,
\frac{1}{\alpha\beta\gamma}
\int_{\omega(\varepsilon;\eta)}
\frac{
e^{s^{1/(\alpha\beta\gamma)}}
\,s^{\frac{\alpha+\beta+\gamma+1-\mu}{\alpha\beta\gamma}-1}
}
{
(s^{1/\beta\gamma}-x)(s^{1/\alpha\gamma}-y)(s^{1/\alpha\beta}-z)
}
\,ds.\end{aligned}\end{equation}

\begin{remark} Analogously, for  $x\in\Omega^{(+)}(\varepsilon_{\alpha};\eta_\alpha),$
\,
$y\in\Omega^{(-)}(\varepsilon_{\beta};\eta_\beta)$,\, $z\in\Omega^{(-)}(\varepsilon_{\gamma};\eta_\gamma)$ or $x\in\Omega^{(-)}(\varepsilon_{\alpha};\eta_\alpha),$
\,
$y\in\Omega^{(+)}(\varepsilon_{\beta};\eta_\beta)$,\, $z\in\Omega^{(-)}(\varepsilon_{\gamma};\eta_\gamma)$
we obtain the following integral representations, respectively 
$$\begin{aligned}
E_{\alpha,\beta,\gamma}(x,y,z;\mu)
&=\frac{1}{\alpha}\frac{
e^{x^{1/\alpha}}
\,x^{\frac{\beta+\gamma+1-\mu}{\alpha}}}
{(x^{\beta/\alpha}-y)(x^{\gamma/\alpha}-z)}\\
&+
\frac{1}{2\pi i}\,
\frac{1}{\alpha\beta\gamma}
\int_{\omega(\varepsilon;\eta)}
\frac{
e^{s^{1/(\alpha\beta\gamma)}}
\,s^{\frac{\alpha+\beta+\gamma+1-\mu}{\alpha\beta\gamma}-1}
}
{
(s^{1/\beta\gamma}-x)(s^{1/\alpha\gamma}-y)(s^{1/\alpha\beta}-z)
}
\,ds,\end{aligned}
$$
and 
$$\begin{aligned}
E_{\alpha,\beta,\gamma}(x,y,z;\mu)
&=\frac{1}{\beta}\frac{
e^{y^{1/\beta}}
\,y^{\frac{\alpha+\gamma+1-\mu}{\beta}}}
{(y^{\alpha/\beta}-x)(y^{\gamma/\beta}-z)}\\
&+
\frac{1}{2\pi i}\,
\frac{1}{\alpha\beta\gamma}
\int_{\omega(\varepsilon;\eta)}
\frac{
e^{s^{1/(\alpha\beta\gamma)}}
\,s^{\frac{\alpha+\beta+\gamma+1-\mu}{\alpha\beta\gamma}-1}
}
{
(s^{1/\beta\gamma}-x)(s^{1/\alpha\gamma}-y)(s^{1/\alpha\beta}-z)
}
\,ds.\end{aligned}
$$\end{remark}

Second, suppose that both $y$ and $z$ lie to the right of the corresponding Hankel contours, that is,
$
y\in\Omega^{(+)}(\varepsilon_\beta;\eta_\beta),
\,\,
z\in\Omega^{(+)}(\varepsilon_\gamma;\eta_\gamma).
$
From \eqref{3.12}, it remains to analyze the remaining contour integral. Applying the change of variables
$
s^{1/\alpha\gamma}=t,
$
we obtain
\begin{equation}\label{3.13}\begin{aligned}
E_{\alpha,\beta,\gamma}(x,y,z;\mu)
&=\frac{1}{\gamma}\frac{
e^{z^{1/\gamma}}
\,z^{\frac{\alpha+\beta+1-\mu}{\gamma}}}
{(z^{\alpha/\gamma}-x)(z^{\beta/\gamma}-y)}\\
&+
\frac{1}{2\pi i}\,
\frac{1}{\beta}
\int_{\omega(\varepsilon_\beta;\eta_\beta)}
\frac{
e^{t^{1/\beta}}
\,t^{\frac{\alpha+\gamma+1-\mu}{\beta}}
}
{
(t^{\alpha/\beta}-x)(t-y)(t^{\gamma/\beta}-z)
}
\,dt.\end{aligned}\end{equation}
Under the transformation
$
s=t^{\alpha\gamma},
$
the Hankel contour
$
\omega(\varepsilon;\eta)
$
is mapped onto the Hankel contour
$
\omega(\varepsilon_\beta;\eta_\beta).
$ Choose an arbitrary number $\widetilde{\varepsilon}_{\beta}$ satisfying
$\widetilde{\varepsilon}_{\beta}>|y|$. Then
$
y\in\Omega^{(-)}(\widetilde{\varepsilon}_{\beta};\eta_\beta).
$
Furthermore, setting
$
\widetilde{\varepsilon}_{\alpha}
=\widetilde{\varepsilon}_{\beta}^{\frac{\alpha}{\beta}},
\,\,
\widetilde{\varepsilon}_{\gamma}
=\widetilde{\varepsilon}_{\beta}^{\frac{\gamma}{\beta}},
$
it follows that
$
x\in\Omega^{(-)}(\widetilde{\varepsilon}_{\alpha};\eta_\alpha),
\,\,
z\in\Omega^{(-)}(\widetilde{\varepsilon}_{\gamma};\eta_\gamma).
$ Hence, applying the integral representation \eqref{3.13}, we arrive at
\begin{equation}\label{3.14}\begin{aligned}
E_{\alpha,\beta,\gamma}(x,y,z;\mu)
&=\frac{1}{\gamma}\frac{
e^{z^{1/\gamma}}
\,z^{\frac{\alpha+\beta+1-\mu}{\gamma}}}
{(z^{\alpha/\gamma}-x)(z^{\beta/\gamma}-y)}\\
&+
\frac{1}{2\pi i}\,
\frac{1}{\beta}
\int_{\omega(\widetilde{\varepsilon}_{\beta};\eta_\beta)}
\frac{
e^{t^{1/\beta}}
\,t^{\frac{\alpha+\gamma+1-\mu}{\beta}}
}
{
(t^{\alpha/\beta}-x)(t-y)(t^{\gamma/\beta}-z)
}
\,dt.\end{aligned}\end{equation}
Moreover, if
$
\varepsilon_\beta<|y|<\widetilde{\varepsilon}_{\beta},
$
then
$
|\arg y|<\eta_\beta,
$
and, by Cauchy's theorem,
\begin{equation}\label{3.15}\begin{aligned}
E_{\alpha,\beta,\gamma}(x,y,z;\mu)
&=
\frac{1}{2\pi i}\,
\frac{1}{\beta}
\int_{\omega(\widetilde{\varepsilon}_{\beta};\eta_\beta)-\omega(\varepsilon_{\beta};\eta_\beta)}
\frac{
e^{t^{1/\beta}}
\,t^{\frac{\alpha+\gamma+1-\mu}{\beta}}
}
{
(t^{\alpha/\beta}-x)(t-y)(t^{\gamma/\beta}-z)
}
\,dt\\
&=\frac{1}{\beta}\frac{
e^{y^{1/\beta}}
\,y^{\frac{\alpha+\gamma+1-\mu}{\beta}}}
{(y^{\alpha/\beta}-x)(y^{\gamma/\beta}-z)}\end{aligned}
\end{equation}
Combining  \eqref{3.15}, \eqref{3.14}, \eqref{3.13} and \eqref{3.12}, we obtain the following integral representation 
\begin{equation}\label{3.16}\begin{aligned}
E_{\alpha,\beta,\gamma}(x,y,z;\mu)
&=\frac{1}{\gamma}\frac{
e^{z^{1/\gamma}}
\,z^{\frac{\alpha+\beta+1-\mu}{\gamma}}}
{(z^{\alpha/\gamma}-x)(z^{\beta/\gamma}-y)}+\frac{1}{\beta}\frac{
e^{y^{1/\beta}}
\,y^{\frac{\alpha+\gamma+1-\mu}{\beta}}}
{(y^{\alpha/\beta}-x)(y^{\gamma/\beta}-z)}\\
&+
\frac{1}{2\pi i}\,
\frac{1}{\alpha\beta\gamma}
\int_{\omega(\varepsilon;\eta)}
\frac{
e^{s^{1/(\alpha\beta\gamma)}}
\,s^{\frac{\alpha+\beta+\gamma+1-\mu}{\alpha\beta\gamma}-1}
}
{
(s^{1/\beta\gamma}-x)(s^{1/\alpha\gamma}-y)(s^{1/\alpha\beta}-z)
}
\,ds.\end{aligned}\end{equation}

\begin{remark} Analogously, for  $x\in\Omega^{(+)}(\varepsilon_{\alpha};\eta_\alpha),$
\,
$y\in\Omega^{(+)}(\varepsilon_{\beta};\eta_\beta)$,\, $z\in\Omega^{(-)}(\varepsilon_{\gamma};\eta_\gamma)$ or $x\in\Omega^{(+)}(\varepsilon_{\alpha};\eta_\alpha),$
\,
$y\in\Omega^{(-)}(\varepsilon_{\beta};\eta_\beta)$,\, $z\in\Omega^{(+)}(\varepsilon_{\gamma};\eta_\gamma)$
we obtain the following integral representations, respectively 
$$\begin{aligned}
E_{\alpha,\beta,\gamma}(x,y,z;\mu)
&=\frac{1}{\alpha}\frac{
e^{x^{1/\alpha}}
\,x^{\frac{\beta+\gamma+1-\mu}{\alpha}}}
{(x^{\beta/\alpha}-y)(x^{\gamma/\alpha}-z)}+\frac{1}{\beta}\frac{
e^{y^{1/\beta}}
\,y^{\frac{\alpha+\gamma+1-\mu}{\beta}}}
{(y^{\alpha/\beta}-x)(y^{\gamma/\beta}-z)}\\
&+
\frac{1}{2\pi i}\,
\frac{1}{\alpha\beta\gamma}
\int_{\omega(\varepsilon;\eta)}
\frac{
e^{s^{1/(\alpha\beta\gamma)}}
\,s^{\frac{\alpha+\beta+\gamma+1-\mu}{\alpha\beta\gamma}-1}
}
{
(s^{1/\beta\gamma}-x)(s^{1/\alpha\gamma}-y)(s^{1/\alpha\beta}-z)
}
\,ds,\end{aligned}
$$
and 
$$\begin{aligned}
E_{\alpha,\beta,\gamma}(x,y,z;\mu)
&=\frac{1}{\alpha}\frac{
e^{x^{1/\alpha}}
\,x^{\frac{\beta+\gamma+1-\mu}{\alpha}}}
{(x^{\beta/\alpha}-y)(x^{\gamma/\alpha}-z)}+\frac{1}{\gamma}\frac{
e^{z^{1/\gamma}}
\,z^{\frac{\alpha+\beta+1-\mu}{\gamma}}}
{(z^{\alpha/\gamma}-x)(z^{\beta/\gamma}-y)}\\
&+
\frac{1}{2\pi i}\,
\frac{1}{\alpha\beta\gamma}
\int_{\omega(\varepsilon;\eta)}
\frac{
e^{s^{1/(\alpha\beta\gamma)}}
\,s^{\frac{\alpha+\beta+\gamma+1-\mu}{\alpha\beta\gamma}-1}
}
{
(s^{1/\beta\gamma}-x)(s^{1/\alpha\gamma}-y)(s^{1/\alpha\beta}-z)
}
\,ds.\end{aligned}
$$\end{remark}

Third, all $x$, $y$ and $z$ points lie to the right of the corresponding Hankel contours, that is,
$
y\in\Omega^{(+)}(\varepsilon_\beta;\eta_\beta),
\,\,
z\in\Omega^{(+)}(\varepsilon_\gamma;\eta_\gamma),\,\,x\in\Omega^{(+)}(\varepsilon_\alpha;\eta_\alpha)
$.
From \eqref{3.16}, it remains to analyze the remaining contour integral. Applying the change of variables
$
s^{1/\beta\gamma}=t,
$
we obtain
\begin{equation}\label{3.17}\begin{aligned}
E_{\alpha,\beta,\gamma}(x,y,z;\mu)
&=\frac{1}{\gamma}\frac{
e^{z^{1/\gamma}}
\,z^{\frac{\alpha+\beta+1-\mu}{\gamma}}}
{(z^{\alpha/\gamma}-x)(z^{\beta/\gamma}-y)}+\frac{1}{\beta}\frac{
e^{y^{1/\beta}}
\,y^{\frac{\alpha+\gamma+1-\mu}{\beta}}}
{(y^{\alpha/\beta}-x)(y^{\gamma/\beta}-z)}\\
&+
\frac{1}{2\pi i}\,
\frac{1}{\alpha}
\int_{\omega(\varepsilon_\alpha;\eta_\alpha)}
\frac{
e^{t^{1/\alpha}}
\,t^{\frac{\beta+\gamma+1-\mu}{\alpha}}
}
{
(t-x)(t^{\beta/\alpha}-y)(t^{\gamma/\alpha}-z)
}
\,dt.\end{aligned}\end{equation}
Under the transformation
$
s=t^{\beta\gamma},
$
the Hankel contour
$
\omega(\varepsilon;\eta)
$
is mapped onto the Hankel contour
$
\omega(\varepsilon_\alpha;\eta_\alpha).
$ Choose an arbitrary number $\widetilde{\varepsilon}_{\alpha}$ satisfying
$\widetilde{\varepsilon}_{\alpha}>|x|$. Then
$
x\in\Omega^{(-)}(\widetilde{\varepsilon}_{\alpha};\eta_\alpha).
$
Furthermore, setting
$
\widetilde{\varepsilon}_{\beta}
=\widetilde{\varepsilon}_{\alpha}^{\frac{\beta}{\alpha}},
\,\,
\widetilde{\varepsilon}_{\gamma}
=\widetilde{\varepsilon}_{\alpha}^{\frac{\gamma}{\alpha}},
$
it follows that
$
y\in\Omega^{(-)}(\widetilde{\varepsilon}_{\beta};\eta_\beta),
\,\,
z\in\Omega^{(-)}(\widetilde{\varepsilon}_{\gamma};\eta_\gamma).
$ Hence, applying the integral representation \eqref{3.17}, we arrive at
\begin{equation}\label{3.18}\begin{aligned}
E_{\alpha,\beta,\gamma}(x,y,z;\mu)
&=\frac{1}{\gamma}\frac{
e^{z^{1/\gamma}}
\,z^{\frac{\alpha+\beta+1-\mu}{\gamma}}}
{(z^{\alpha/\gamma}-x)(z^{\beta/\gamma}-y)}+\frac{1}{\beta}\frac{
e^{y^{1/\beta}}
\,y^{\frac{\alpha+\gamma+1-\mu}{\beta}}}
{(y^{\alpha/\beta}-x)(y^{\gamma/\beta}-z)}\\
&+
\frac{1}{2\pi i}\,
\frac{1}{\alpha}
\int_{\omega(\widetilde{\varepsilon}_{\alpha};\eta_\alpha)}
\frac{
e^{t^{1/\alpha}}
\,t^{\frac{\beta+\gamma+1-\mu}{\alpha}}
}
{
(t-x)(t^{\beta/\alpha}-y)(t^{\gamma/\alpha}-z)
}
\,dt.\end{aligned}\end{equation}
Moreover, if
$
\varepsilon_\alpha<|x|<\widetilde{\varepsilon}_{\alpha},
$
then
$
|\arg x|<\eta_\alpha,
$
and, by Cauchy's theorem,
\begin{equation}\label{3.19}\begin{aligned}
E_{\alpha,\beta,\gamma}(x,y,z;\mu)
&=
\frac{1}{2\pi i}\,
\frac{1}{\alpha}
\int_{\omega(\widetilde{\varepsilon}_{\alpha};\eta_\alpha)-\omega(\varepsilon_{\alpha};\eta_\alpha)}
\frac{
e^{t^{1/\alpha}}
\,t^{\frac{\beta+\gamma+1-\mu}{\alpha}}
}
{
(t-x)(t^{\beta/\alpha}-y)(t^{\gamma/\alpha}-z)
}
\,dt\\
&=\frac{1}{\alpha}\frac{
e^{x^{1/\alpha}}
\,x^{\frac{\beta+\gamma+1-\mu}{\alpha}}}
{(x^{\beta/\alpha}-y)(x^{\gamma/\alpha}-z)}\end{aligned}
\end{equation}
Combining \eqref{3.19}, \eqref{3.18}, \eqref{3.17}, and \eqref{3.16}, we obtain
 the following integral representation 
\begin{equation}\label{3.20}\begin{aligned}
&E_{\alpha,\beta,\gamma}(x,y,z;\mu)\\
&=\frac{1}{\alpha}\frac{
e^{x^{1/\alpha}}
\,x^{\frac{\beta+\gamma+1-\mu}{\alpha}}}
{(x^{\beta/\alpha}-y)(x^{\gamma/\alpha}-z)}+\frac{1}{\beta}\frac{
e^{y^{1/\beta}}
\,y^{\frac{\alpha+\gamma+1-\mu}{\beta}}}
{(y^{\alpha/\beta}-x)(y^{\gamma/\beta}-z)}+\frac{1}{\gamma}\frac{
e^{z^{1/\gamma}}
\,z^{\frac{\alpha+\beta+1-\mu}{\gamma}}}
{(z^{\alpha/\gamma}-x)(z^{\beta/\gamma}-y)}\\
&+
\frac{1}{2\pi i}\,
\frac{1}{\alpha\beta\gamma}
\int_{\omega(\varepsilon;\eta)}
\frac{
e^{s^{1/(\alpha\beta\gamma)}}
\,s^{\frac{\alpha+\beta+\gamma+1-\mu}{\alpha\beta\gamma}-1}
}
{
(s^{1/\beta\gamma}-x)(s^{1/\alpha\gamma}-y)(s^{1/\alpha\beta}-z)
}
\,ds.\end{aligned}\end{equation}
Thus Lemma \ref{l3.2} is proved. \qed
\end{proof}

\section{Asymptotic behaviors of the three-variable Mittag-Leffler function }\label{sec4}
\begin{theorem}\label{thm4.1}
    Let $0<\alpha,\beta,\gamma<2$ and $\alpha\beta\gamma<2$. Let $\mu$ be any complex number and
$\tau$ be any real number satisfying the inequalities
\begin{equation}\label{4.1}
\frac{\pi\alpha\beta\gamma}{2}<\tau< \min\{\pi,\pi\alpha\beta\gamma\}.
\end{equation}

Then, for all integers $r_{\alpha},r_{\beta,}r_{\gamma}\geq 1$, the function
$E_{\alpha,\beta,\gamma}(x,y,z;\mu)$ verifies the following asymptotic formulas
drawn from its integral representations whenever $|x|\to\infty$,\,\,
$|y|\to\infty$ and $|z|\to\infty$.
\begin{itemize}
    \item If $\tau/\beta\gamma \leq |\arg x| \leq \pi$,\,\, 
$\tau/\alpha\gamma \leq |\arg y| \leq \pi$ and $\tau/\alpha\beta \leq |\arg z| \leq \pi$ then
$$E_{\alpha,\beta,\gamma}(x,y,z;\mu)
=-\sum_{n=1}^{r_{\alpha}}
\sum_{m=1}^{r_{\beta}}
\sum_{p=1}^{r_{\gamma}}
\frac{x^{-n}y^{-m}z^{-p}}{\Gamma(\mu-n\alpha-m\beta-p\gamma)}$$
$$+o\!\left(|x|^{-r_{\alpha}}|xyz|^{-1}\right)
+
o\!\left(|y|^{-r_{\beta}}|xyz|^{-1}\right)
+
o\left(|z|^{-r_{\gamma}}|xyz|^{-1}\right).$$

\item If $ |\arg x| \leq \tau/\beta\gamma $,\,\, 
$  |\arg y| \leq \tau/\alpha\gamma$ and $  |\arg z| \leq \tau/\alpha\beta$ then

$$\begin{aligned}
E_{\alpha,\beta,\gamma}(x,y,z;\mu)&=\frac{1}{\alpha}\frac{
e^{x^{1/\alpha}}
\,x^{\frac{\beta+\gamma+1-\mu}{\alpha}}}
{(x^{\beta/\alpha}-y)(x^{\gamma/\alpha}-z)}+\frac{1}{\beta}\frac{
e^{y^{1/\beta}}
\,y^{\frac{\alpha+\gamma+1-\mu}{\beta}}}
{(y^{\alpha/\beta}-x)(y^{\gamma/\beta}-z)}\\
&+
\frac{1}{\gamma}\frac{
e^{z^{1/\gamma}}
\,z^{\frac{\alpha+\beta+1-\mu}{\gamma}}}
{(z^{\alpha/\gamma}-x)(z^{\beta/\gamma}-y)}
-\sum_{n=1}^{r_{\alpha}}
\sum_{m=1}^{r_{\beta}}
\sum_{p=1}^{r_{\gamma}}
\frac{x^{-n}y^{-m}z^{-p}}{\Gamma(\mu-n\alpha-m\beta-p\gamma)}\\
&+o\!\left(|x|^{-r_{\alpha}}|xyz|^{-1}\right)
+
o\!\left(|y|^{-r_{\beta}}|xyz|^{-1}\right)
+
o\left(|z|^{-r_{\gamma}}|xyz|^{-1}\right).\end{aligned}$$
\end{itemize}
\end{theorem}

\begin{proof} The denominator in \eqref{three-variable integral representation} can be rewritten as follows:
\begin{equation}\label{3 variable denominator}\begin{aligned}\frac{1}
{\left(s^{\frac{1}{\beta\gamma}}-x\right)
 \left(s^{\frac{1}{\alpha\gamma}}-y\right)
 \left(s^{\frac{1}{\alpha\beta}}-z\right)}&=-
\sum_{n=1}^{r_{\alpha}}
\sum_{m=1}^{r_{\beta}}
\sum_{p=1}^{r_{\gamma}}
\frac{
s^{\frac{n-1}{\beta\gamma}
+\frac{m-1}{\alpha\gamma}
+\frac{p-1}{\alpha\beta}}
}
{x^n y^m z^p}
\\&+
\frac{
x^{r_{\alpha}}y^{r_{\beta}}
s^{\frac{r_{\gamma}}{\alpha\beta}}
+x^{r_{\alpha}}z^{r_{\gamma}}
s^{\frac{r_{\beta}}{\alpha\gamma}}
+y^{r_{\beta}}z^{r_{\gamma}}
s^{\frac{r_{\alpha}}{\beta\gamma}}
}
{
x^{r_{\alpha}}
y^{r_{\beta}}
z^{r_{\gamma}}
\left(s^{\frac{1}{\beta\gamma}}-x\right)
\left(s^{\frac{1}{\alpha\gamma}}-y\right)
\left(s^{\frac{1}{\alpha\beta}}-z\right)
}
\\&-
\frac{
x^{r_{\alpha}}
s^{\frac{r_{\beta}}{\alpha\gamma}
+\frac{r_{\gamma}}{\alpha\beta}}
+y^{r_{\beta}}
s^{\frac{r_{\alpha}}{\beta\gamma}
+\frac{r_{\gamma}}{\alpha\beta}}
+z^{r_{\gamma}}
s^{\frac{r_{\alpha}}{\beta\gamma}
+\frac{r_{\beta}}{\alpha\gamma}}
}
{
x^{r_{\alpha}}
y^{r_{\beta}}
z^{r_{\gamma}}
\left(s^{\frac{1}{\beta\gamma}}-x\right)
\left(s^{\frac{1}{\alpha\gamma}}-y\right)
\left(s^{\frac{1}{\alpha\beta}}-z\right)
}
\\&+
\frac{
s^{\frac{r_{\alpha}}{\beta\gamma}
+\frac{r_{\beta}}{\alpha\gamma}
+\frac{r_{\gamma}}{\alpha\beta}}
}
{
x^{r_{\alpha}}
y^{r_{\beta}}
z^{r_{\gamma}}
\left(s^{\frac{1}{\beta\gamma}}-x\right)
\left(s^{\frac{1}{\alpha\gamma}}-y\right)
\left(s^{\frac{1}{\alpha\beta}}-z\right)
}.\end{aligned}\end{equation}

The proof of \eqref{3 variable denominator} is presented in Section~\ref{sec5} (see \eqref{5.10}-\eqref{5.15}).

Let us
choose a real number $\theta$ satisfying
\begin{equation}\label{1.5}
\frac{\pi\alpha\beta\gamma}{2}
<
\theta
<
\tau
<
\min\{\pi,\pi\alpha\beta\gamma\}.
\end{equation}

Using the corresponding integral representation established in Lemma \ref{l3.1}, we may assume, without loss of generality, that $\varepsilon=1$, since the representation is valid for every $\varepsilon>0$. Substituting the decomposition into \eqref{three-variable integral representation}, we obtain

$$E_{\alpha,\beta,\gamma}(x,y,z;\mu)$$
$$
=\sum_{n=1}^{r_{\alpha}}
\sum_{m=1}^{r_{\beta}}
\sum_{p=1}^{r_{\gamma}}
\frac{1}{2\pi i}\,
\frac{1}{\alpha\beta\gamma}
\left(\int_{\omega(1;\theta)}
e^{s^{1/(\alpha\beta\gamma)}}s^{\frac{\alpha+\beta+\gamma+1-\mu}{\alpha\beta\gamma}-1+\frac{n-1}{\beta\gamma}+\frac{m-1}{\alpha\gamma}+\frac{p-1}{\alpha\beta}}
\,ds\right)x^{-n}y^{-m}z^{-p}$$
$$+\frac{1}{2\pi i}\,
\frac{1}{\alpha\beta\gamma}
\left(\int_{\omega(1;\theta)}
e^{s^{1/(\alpha\beta\gamma)}}s^{\frac{\alpha+\beta+\gamma+1-\mu}{\alpha\beta\gamma}-1}\frac{
x^{r_{\alpha}}y^{r_{\beta}}
s^{\frac{r_{\gamma}}{\alpha\beta}}
+x^{r_{\alpha}}z^{r_{\gamma}}
s^{\frac{r_{\beta}}{\alpha\gamma}}
+y^{r_{\beta}}z^{r_{\gamma}}
s^{\frac{r_{\alpha}}{\beta\gamma}}
}
{
x^{r_{\alpha}}
y^{r_{\beta}}
z^{r_{\gamma}}
\left(s^{\frac{1}{\beta\gamma}}-x\right)
\left(s^{\frac{1}{\alpha\gamma}}-y\right)
\left(s^{\frac{1}{\alpha\beta}}-z\right)
}ds\right)$$
$$-\frac{1}{2\pi i}\,
\frac{1}{\alpha\beta\gamma}
\left(\int_{\omega(1;\theta)}
e^{s^{1/(\alpha\beta\gamma)}}s^{\frac{\alpha+\beta+\gamma+1-\mu}{\alpha\beta\gamma}-1}\frac{
x^{r_{\alpha}}
s^{\frac{r_{\beta}}{\alpha\gamma}
+\frac{r_{\gamma}}{\alpha\beta}}
+y^{r_{\beta}}
s^{\frac{r_{\alpha}}{\beta\gamma}
+\frac{r_{\gamma}}{\alpha\beta}}
+z^{r_{\gamma}}
s^{\frac{r_{\alpha}}{\beta\gamma}
+\frac{r_{\beta}}{\alpha\gamma}}
}
{
x^{r_{\alpha}}
y^{r_{\beta}}
z^{r_{\gamma}}
\left(s^{\frac{1}{\beta\gamma}}-x\right)
\left(s^{\frac{1}{\alpha\gamma}}-y\right)
\left(s^{\frac{1}{\alpha\beta}}-z\right)
}ds\right)$$
$$+\frac{1}{2\pi i}\,
\frac{1}{\alpha\beta\gamma}
\left(\int_{\omega(1;\theta)}
e^{s^{1/(\alpha\beta\gamma)}}s^{\frac{\alpha+\beta+\gamma+1-\mu}{\alpha\beta\gamma}-1}\frac{
s^{\frac{r_{\alpha}}{\beta\gamma}
+\frac{r_{\beta}}{\alpha\gamma}
+\frac{r_{\gamma}}{\alpha\beta}}
}
{
x^{r_{\alpha}}
y^{r_{\beta}}
z^{r_{\gamma}}
\left(s^{\frac{1}{\beta\gamma}}-x\right)
\left(s^{\frac{1}{\alpha\gamma}}-y\right)
\left(s^{\frac{1}{\alpha\beta}}-z\right)
}ds\right).$$

By simplifying the first expression, we obtain
$$-\sum_{n=1}^{r_{\alpha}}
\sum_{m=1}^{r_{\beta}}
\sum_{p=1}^{r_{\gamma}}
\frac{1}{2\pi i}\,
\frac{1}{\alpha\beta\gamma}
\left(\int_{\omega(1;\theta)}
e^{s^{1/(\alpha\beta\gamma)}}s^{\frac{1-\mu}{\alpha\beta\gamma}-1+\frac{n}{\beta\gamma}+\frac{m}{\alpha\gamma}+\frac{p}{\alpha\beta}}
\,ds\right)x^{-n}y^{-m}z^{-p}$$

$$=-\sum_{n=1}^{r_{\alpha}}\sum_{m=1}^{r_{\beta}}
\sum_{p=1}^{r_{\gamma}}
\frac{1}{2\pi i}\,
\frac{1}{\alpha\beta\gamma}
\left(\int_{\omega(1;\theta)}
e^{s^{1/(\alpha\beta\gamma)}}s^{\frac{1-(\mu-\alpha n-\beta m-\gamma p)}{\alpha\beta\gamma}-1}
\,ds\right)x^{-n}y^{-m}z^{-p}.$$

Using the Hankel integral representation of the reciprocal Gamma function (see \cite{Podlubny}, p. 16),  
\begin{equation}\label{gamma integral representation}
\frac{1}{\Gamma(a)} =
\frac{1}{2 b \pi i}
\int_{\omega(1;\theta)}
e^{s^{1/b}}\,
s^{\frac{1-a}{b}-1}\, ds,
\end{equation}
we get

$$E_{\alpha,\beta,\gamma}(x,y,z;\mu)
=-\sum_{n=1}^{r_{\alpha}}
\sum_{m=1}^{r_{\beta}}
\sum_{p=1}^{r_{\gamma}}
\frac{x^{-n}y^{-m}z^{-p}}{\Gamma(\mu-n\alpha-m\beta-p\gamma)}+I_1-I_2+I_3,
$$
where
$$I_1=\frac{1}{2\pi i}\,
\frac{1}{\alpha\beta\gamma}
\left(\int_{\omega(1;\theta)}
e^{s^{1/(\alpha\beta\gamma)}}s^{\frac{\alpha+\beta+\gamma+1-\mu}{\alpha\beta\gamma}-1}\frac{
x^{r_{\alpha}}y^{r_{\beta}}
s^{\frac{r_{\gamma}}{\alpha\beta}}
+x^{r_{\alpha}}z^{r_{\gamma}}
s^{\frac{r_{\beta}}{\alpha\gamma}}
+y^{r_{\beta}}z^{r_{\gamma}}
s^{\frac{r_{\alpha}}{\beta\gamma}}
}
{
x^{r_{\alpha}}
y^{r_{\beta}}
z^{r_{\gamma}}
\left(s^{\frac{1}{\beta\gamma}}-x\right)
\left(s^{\frac{1}{\alpha\gamma}}-y\right)
\left(s^{\frac{1}{\alpha\beta}}-z\right)
}ds\right),$$
$$
I_2=\frac{1}{2\pi i}\,
\frac{1}{\alpha\beta\gamma}
\left(\int_{\omega(1;\theta)}
e^{s^{1/(\alpha\beta\gamma)}}s^{\frac{\alpha+\beta+\gamma+1-\mu}{\alpha\beta\gamma}-1}\frac{
x^{r_{\alpha}}
s^{\frac{r_{\beta}}{\alpha\gamma}
+\frac{r_{\gamma}}{\alpha\beta}}
+y^{r_{\beta}}
s^{\frac{r_{\alpha}}{\beta\gamma}
+\frac{r_{\gamma}}{\alpha\beta}}
+z^{r_{\gamma}}
s^{\frac{r_{\alpha}}{\beta\gamma}
+\frac{r_{\beta}}{\alpha\gamma}}
}
{
x^{r_{\alpha}}
y^{r_{\beta}}
z^{r_{\gamma}}
\left(s^{\frac{1}{\beta\gamma}}-x\right)
\left(s^{\frac{1}{\alpha\gamma}}-y\right)
\left(s^{\frac{1}{\alpha\beta}}-z\right)
}ds\right),$$
$$
I_3=\frac{1}{2\pi i}\,
\frac{1}{\alpha\beta\gamma}
\left(\int_{\omega(1;\theta)}
e^{s^{1/(\alpha\beta\gamma)}}s^{\frac{\alpha+\beta+\gamma+1-\mu}{\alpha\beta\gamma}-1}\frac{
s^{\frac{r_{\alpha}}{\beta\gamma}
+\frac{r_{\beta}}{\alpha\gamma}
+\frac{r_{\gamma}}{\alpha\beta}}
}
{
x^{r_{\alpha}}
y^{r_{\beta}}
z^{r_{\gamma}}
\left(s^{\frac{1}{\beta\gamma}}-x\right)
\left(s^{\frac{1}{\alpha\gamma}}-y\right)
\left(s^{\frac{1}{\alpha\beta}}-z\right)
}ds\right).$$

Provided that
$
\cfrac{\tau}{\beta\gamma}\leq|\arg x|\le\pi,
$
for sufficiently large $|x|$, one has
\[
\min_{s\in\omega(1;\theta)}
\left|s^{1/\beta\gamma}-x\right|
= |x|
\sin\left(\cfrac{\tau}{\beta\gamma}-\cfrac{\theta}{\beta\gamma}\right)=|x|\sin\left(\cfrac{\tau-\theta}{\beta\gamma}\right).
\]
Similarly, for sufficiently large $|y|$ and $|z|$, provided that
\[
\frac{\tau}{\alpha\gamma}\leq|\arg y|\le\pi,
\qquad
\frac{\tau}{\alpha\beta}\leq|\arg z|\le\pi,
\]
we have
\[
\min_{s\in\omega(1;\theta)}
\left|s^{1/\alpha\gamma}-y\right|
= |y|
\sin\left(\cfrac{\tau}{\alpha\gamma}-\cfrac{\theta}{\alpha\gamma}\right)=|y|\sin\left(\cfrac{\tau-\theta}{\alpha\gamma}\right),
\]
and
\[
\min_{s\in\omega(1;\theta)}
\left|s^{1/\alpha\beta}-z\right|
= |z|
\sin\left(\cfrac{\tau}{\alpha\beta}-\cfrac{\theta}{\alpha\beta}\right)=|z|\sin\left(\cfrac{\tau-\theta}{\alpha\beta}\right).
\]

Then, we obtain
\begin{equation}\label{1.6}
\begin{aligned}
&|I_1|
\le
\frac{
1
}
{
2\pi\alpha\beta\gamma
\sin\left(\cfrac{\tau-\theta}{\beta\gamma}\right)\sin\left(\cfrac{\tau-\theta}{\alpha\gamma}\right)\sin\left(\cfrac{\tau-\theta}{\alpha\beta}\right)
}
\\&\times\Bigg(\int_{\omega(1;\theta)}
\left|
e^{s^{1/(\alpha\beta\gamma)}}
\right|
\left|
s^{
\frac{\alpha+\beta+\gamma+1-\mu}{\alpha\beta\gamma}
-1
+\frac{r_{\alpha}}{\beta\gamma}
}
\right||x|^{-r_{\alpha}-1}|y|^{-1}|z|^{-1}
|ds|\\
&+\int_{\omega(1;\theta)}
\left|
e^{s^{1/(\alpha\beta\gamma)}}
\right|
\left|
s^{
\frac{\alpha+\beta+\gamma+1-\mu}{\alpha\beta\gamma}
-1
+\frac{r_{\beta}}{\alpha\gamma}
}
\right||x|^{-1}|y|^{-r_{\beta}-1}|z|^{-1}
|ds|\\
&+\int_{\omega(1;\theta)}
\left|
e^{s^{1/(\alpha\beta\gamma)}}
\right|
\left|
s^{
\frac{\alpha+\beta+\gamma+1-\mu}{\alpha\beta\gamma}
-1
+\frac{r_{\gamma}}{\alpha\beta}
}
\right||x|^{-1}|y|^{-1}|z|^{-r_{\gamma}-1}
|ds|\Bigg).
\end{aligned}\end{equation}

Moreover, since the rays defined by
$
S_{\pm\theta}
=
\left\{
\arg s = \pm \theta,\; |s| \ge 1
\right\}
$
belong to the contour $\omega(1;\theta)$, all the integrals in \eqref{1.6} are convergent; since
\[
\left|e^{s^{1/(\alpha\beta\gamma)}}\right|
=
\exp\!\left(
|s|^{1/(\alpha\beta\gamma)}
\cos\!\left(\frac{\theta}{\alpha\beta\gamma}\right)
\right).
\]

Now, according to inequalities \eqref{1.5}, we have that $\cos\!\left(\cfrac{\theta}{\alpha\beta\gamma}\right)<0$. Thus,
$$I_1=O\left(|x|^{-r_{\alpha}-1}|y|^{-1}|z|^{-1}\right)+O\left(|x|^{-1}|y|^{-r_{\beta}-1}|z|^{-1}\right)+O\left(|x|^{-1}|y|^{-1}|z|^{-r_{\gamma}-1}\right).$$

By proceeding in the same way for $I_2$ and $I_3$, we arrive at
\[
I_2=O\left(|x|^{-r_{\alpha}-1}|y|^{-r_{\beta}-1}|z|^{-1}\right)+O\left(|x|^{-1}|y|^{-r_{\beta}-1}|z|^{-r_{\gamma}-1}\right)+O\left(|x|^{-r_{\alpha}-1}|y|^{-1}|z|^{-r_{\gamma}-1}\right),
\]
and
\[
I_3=O\left(|x|^{-r_{\alpha}-1}|y|^{-r_{\beta}-1}|z|^{-r_{\gamma}-1}\right).
\]

Since the contributions of \(I_2\) and \(I_3\) are of higher order (i.e., decay faster), we obtain
$$
I_1-I_2+I_3
=
o\!\left(|x|^{-r_{\alpha}-1}|y|^{-1}|z|^{-1}\right)
+
o\!\left(|x|^{-1}|y|^{-r_{\beta}-1}|z|^{-1}\right)
+
o\!\left(|x|^{-1}|y|^{-1}|z|^{-r_{\gamma}-1}\right),
$$
as $|x|,|y|,|z|\to\infty$.

Therefore, as
$|x|,|y|,|z|\to\infty$,
we obtain

$$E_{\alpha,\beta,\gamma}(x,y,z;\mu)
=-\sum_{n=1}^{r_{\alpha}}
\sum_{m=1}^{r_{\beta}}
\sum_{p=1}^{r_{\gamma}}
\frac{x^{-n}y^{-m}z^{-p}}{\Gamma(\mu-n\alpha-m\beta-p\gamma)}$$
$$+o\!\left(|x|^{-r_{\alpha}}|xyz|^{-1}\right)
+
o\!\left(|y|^{-r_{\beta}}|xyz|^{-1}\right)
+
o\left(|z|^{-r_{\gamma}}|xyz|^{-1}\right).$$

If $ |\arg x| \leq \tau/\beta\gamma $,\,\, 
$  |\arg y| \leq \tau/\alpha\gamma$ and $  |\arg z| \leq \tau/\alpha\beta$, we choose a real number $\theta$ satisfying
$$
\frac{\pi\alpha\beta\gamma}{2}
<
\tau
<
\theta
\leq
\min\{\pi,\pi\alpha\beta\gamma\}.
$$

Using the corresponding integral representations established in Lemma \ref{l3.2}, the proof proceeds in the same way as the above one. The only difference is that the corresponding lower estimates

\[
\min_{s\in\omega(1;\theta)}
\left|s^{1/\beta\gamma}-x\right|
= |x|
\sin\left(\cfrac{\theta}{\beta\gamma}-\cfrac{\tau}{\beta\gamma}\right)=|x|\sin\left(\cfrac{\theta-\tau}{\beta\gamma}\right),
\]

\[
\min_{s\in\omega(1;\theta)}
\left|s^{1/\alpha\gamma}-y\right|
= |y|
\sin\left(\cfrac{\theta}{\alpha\gamma}-\cfrac{\tau}{\alpha\gamma}\right)=|y|\sin\left(\cfrac{\theta-\tau}{\alpha\gamma}\right),
\]

\[
\min_{s\in\omega(1;\theta)}
\left|s^{1/\alpha\beta}-z\right|
= |z|
\sin\left(\cfrac{\theta}{\alpha\beta}-\cfrac{\tau}{\alpha\beta}\right)=|z|\sin\left(\cfrac{\theta-\tau}{\alpha\beta}\right).
\]
are used instead.
Then, we get
$$\begin{aligned}
E_{\alpha,\beta,\gamma}(x,y,z;\mu)&=\frac{1}{\alpha}\frac{
e^{x^{1/\alpha}}
\,x^{\frac{\beta+\gamma+1-\mu}{\alpha}}}
{(x^{\beta/\alpha}-y)(x^{\gamma/\alpha}-z)}+\frac{1}{\beta}\frac{
e^{y^{1/\beta}}
\,y^{\frac{\alpha+\gamma+1-\mu}{\beta}}}
{(y^{\alpha/\beta}-x)(y^{\gamma/\beta}-z)}\\
&+
\frac{1}{\gamma}\frac{
e^{z^{1/\gamma}}
\,z^{\frac{\alpha+\beta+1-\mu}{\gamma}}}
{(z^{\alpha/\gamma}-x)(z^{\beta/\gamma}-y)}
-\sum_{n=1}^{r_{\alpha}}
\sum_{m=1}^{r_{\beta}}
\sum_{p=1}^{r_{\gamma}}
\frac{x^{-n}y^{-m}z^{-p}}{\Gamma(\mu-n\alpha-m\beta-p\gamma)}\\
&+o\!\left(|x|^{-r_{\alpha}}|xyz|^{-1}\right)
+
o\!\left(|y|^{-r_{\beta}}|xyz|^{-1}\right)
+
o\left(|z|^{-r_{\gamma}}|xyz|^{-1}\right).\end{aligned}$$
Thus Theorem \ref{thm4.1} is proved. \qed \end{proof}

\section{Integral representations and asymptotic behaviors of the multivariate Mittag-Leffler function}\label{sec5}

First, we introduce the following notations:
$$\beta_l=\cfrac{\prod\limits_{{j=1}}^n\alpha_j}{\alpha_l}, \qquad
\gamma_l=\sum_{j=1}^n\alpha_j-\alpha_l,\qquad l=1,\dots,n.
$$
\begin{lemma}\label{l5.1}
    Let $0 < \alpha_1, \alpha_2,\dots,\alpha_n < 2$ and $\prod\limits_{{j=1}}^n\alpha_j < 2$. Let $\mu$ be any complex number and let $\eta$ satisfy the condition
\begin{equation}\label{5.1}
\frac{\pi}{2}\prod_{{j=1}}^n\alpha_j
<
\eta
\leq
\min\!\Bigl(\pi,\pi\prod_{{j=1}}^n\alpha_j\Bigr).
\end{equation}

If
$
z_j \in \Omega^{(-)}({\varepsilon}_{\alpha_j};\eta_{\alpha_j})
$,
where
$
{\varepsilon}_{\alpha_j}={\varepsilon}^{\frac{1}{\beta_j}}$ and $\eta_{\alpha_j}=\cfrac{\eta}{\beta_j}$, $j=1,\dots,n,$
then the following Hankel integral representation holds:
\begin{equation}\label{multi variable integral representation}
E_{\alpha_1, \dots, \alpha_n }(z_1, \dots, z_n;\mu)
=
\frac{1}{2\pi i}\,
\frac{1}{\prod\limits_{{j=1}}^n\alpha_j}
\int_{\omega(\varepsilon;\eta)}
\cfrac{
e^{s^{1/\prod\limits_{{j=1}}^n\alpha_j}}
\,s^{\frac{\sum\limits_{j=1}^n\alpha_j+1-\mu}{\prod\limits_{j=1}^n\alpha_j}-1}
}
{
\prod\limits_{j=1}^n(s^{1/\beta_j}-z_j)
}
\,ds.
\end{equation}
\end{lemma}

\begin{proof}
We rewrite the $n$-variable Mittag--Leffler function in the following form:
\begin{equation}\label{5.2}
\begin{aligned}
&
E_{\alpha_1,\dots,\alpha_n}(z_1,\dots,z_n;\mu)
=
\sum_{k_1=0}^{\infty}
\sum_{k_2=0}^{\infty}\cdots
\sum_{k_n=0}^{\infty}
\frac{\prod\limits_{j=1}^{n}z_j^{k_j}}
{\Gamma\!\left(\mu+\sum_{j=1}^{n}\alpha_jk_j\right)}\\
&=
\sum_{k_1=0}^{\infty}
\sum_{k_2=0}^{\infty}\cdots\sum_{k_{n-1}=0}^{\infty}z_1^{k_1}z_2^{k_2}\cdots z_{n-1}^{k_{n-1}}
\sum_{k_n=0}^{\infty}
\frac{z_n^{k_n}}
{\Gamma\!\left(\alpha_nk_n+\left(\sum_{j=1}^{n-1}\alpha_jk_j+\mu\right)\right)}\\
&=\sum_{k_1=0}^{\infty}
\sum_{k_2=0}^{\infty}\cdots\sum_{k_{n-1}=0}^{\infty}z_1^{k_1}z_2^{k_2}\cdots z_{n-1}^{k_{n-1}}E_{\alpha_n}(z_n;\alpha_1k_1+\dots\alpha_{n-1}k_{n-1}+\mu).
\end{aligned}\end{equation}

Under the assumptions of Lemma \ref{l5.1}, one can employ the known integral representation of
$E_{\alpha_n}\left(z_n;\alpha_1 k_1+...+\alpha_{n-1}k_{n-1}+\mu\right)$, taking the above
$\varepsilon_{\alpha_n}$ and $\eta_{\alpha_n}$ as the parameters defining the Hankel contour.
This choice is admissible in view of inequalities \eqref{5.1}. Consequently, for
$
z_n\in\Omega^{(-)}(\varepsilon_{\alpha_n};\eta_{\alpha_n}),
$
and provided that
$
\eta_{\alpha_n}=\cfrac{\eta}{\beta_n},
$
the following representation follows from the integral representation of
$E_{\alpha_n}(z_n;\alpha_1k_1+\dots+\alpha_{n-1}k_{n-1}+\mu)$ (see \cite{Podlubny}, p. 30):
$$
E_{\alpha_n}(z_n;\alpha_1k_1+\dots+\alpha_{n-1}k_{n-1}+\mu)
=
\frac{1}{2\pi i}\,
\frac{1}{\alpha_n}
\int_{\omega(\varepsilon_{\alpha_n};\eta_{\alpha_n})}
\frac{
e^{s^{1/\alpha_n}}
\,s^{\frac{1-\alpha_1 k_1-\dots -\alpha_{n-1}k_{n-1}-\mu}{\alpha_n}}
}
{
(s-z_n)
}
\,ds.
$$
We rewrite \eqref{5.2} in the following form:
\begin{equation}\label{5.3}
\begin{aligned}
&E_{\alpha_1,\dots,\alpha_n}(z_1,\dots,z_n;\mu)\\
&=\frac{1}{2\pi i}\,
\frac{1}{\alpha_n}
\int_{\omega(\varepsilon_{\alpha_n};\eta_{\alpha_n})}
\frac{
e^{s^{1/\alpha_n}}
\,s^{\frac{1-\mu}{\alpha_n}}
}
{
(s-z_n)
}\left(\sum_{k_1=0}^{\infty}\left(z_1s^{-\alpha_1/\alpha_n}\right)^{k_1}\cdots\sum_{k_{n-1}=0}^{\infty}\left(z_{n-1}s^{-\alpha_{n-1}/\alpha_n}\right)^{k_{n-1}}\right)
\,ds.
\end{aligned}\end{equation}
Let $|z_1|<\varepsilon_{\alpha_1}$,$\dots$, $|z_{n-1}|<\varepsilon_{\alpha_{n-1}}$. Taking into account the fact that $\varepsilon_{\alpha_1}=\varepsilon^{1/\beta_1}=\left(\varepsilon_{\alpha_n}^{\beta_n}\right)^{1/\beta_1}=\varepsilon_{\alpha_n}^{\alpha_1/\alpha_n}$,$\dots$, $\varepsilon_{\alpha_{n-1}}=\varepsilon^{1/\beta_{n-1}}=\left(\varepsilon_{\alpha_n}^{\beta_{n}}\right)^{1/\beta_{n-1}}=\varepsilon_{\alpha_n}^{\alpha_{n-1}/\alpha_n}$ yields next the inequalities
\begin{equation}\label{5.4}
\sup_{s\in\omega(\varepsilon_{\alpha_n};\eta_{\alpha_n})}
\left|z_1s^{-\alpha_1/\alpha_n}\right|<1,
\end{equation}
$$\vdots$$
 \begin{equation}\label{5.5}
\sup_{s\in\omega(\varepsilon_{\alpha_n};\eta_{\alpha_n})}
\left|z_{n-1}s^{-\alpha_{n-1}/\alpha_n}\right|<1.
\end{equation}

Using \eqref{5.4}--\eqref{5.5} inequalities, we rewrite \eqref{5.3} in the following form:

\begin{equation}\label{5.6}
E_{\alpha_1,\dots,\alpha_n}(z_1,\dots,z_n;\mu)=
\frac{1}{2\pi i}\,
\frac{1}{\alpha_n}
\int_{\omega(\varepsilon_{\alpha_n};\eta_{\alpha_n})}
\frac{
e^{s^{1/\alpha_n}}
\,s^{\frac{\gamma_n+1-\mu}{\alpha_n}}
}
{
\left(\prod\limits_{j=1}^{n-1}(s^{\alpha_j/\alpha_n}-z_j)\right)(s-z_n)
}
\,ds.
\end{equation}
We introduce the notation
$s=\zeta^{\frac{1}{\beta_n}}$, $ds=\cfrac{1}{\beta_n}\zeta^{\frac{1-\beta_n}{\beta_n}}d\zeta$, we get
$$
E_{\alpha_1, \dots, \alpha_n }(z_1, \dots, z_n;\mu)
=
\frac{1}{2\pi i}\,
\frac{1}{\prod\limits_{{j=1}}^n\alpha_j}
\int_{\omega(\varepsilon;\eta)}
\frac{
e^{s^{1/\prod\limits_{{j=1}}^n\alpha_j}}
\,s^{\frac{\sum\limits_{j=1}^n\alpha_j+1-\mu}{\prod\limits_{j=1}^n\alpha_j}-1}
}
{
\prod\limits_{j=1}^n(s^{1/\beta_j}-z_j)
}
\,ds.
$$

The resulting integral is absolutely convergent and defines an analytic function of
\[
z_j\in\Omega^{(-)}(\varepsilon_{\alpha_j};\eta_{\alpha_j}),\qquad
j=1,\dots,n.
\]
Moreover,
$
|z_j|<\varepsilon_{\alpha_j},\,\, j=1,\dots,n,
$
imply that
\[
z_j\in\Omega^{(-)}(\varepsilon_{\alpha_j};\eta_{\alpha_j}),\qquad
j=1,\dots,n.
\]
Therefore, by the principle of analytic continuation, the obtained integral representation remains valid throughout
\[
\Omega^{(-)}(\varepsilon_{\alpha_1};\eta_{\alpha_1})
\times
\cdots
\times
\Omega^{(-)}(\varepsilon_{\alpha_n};\eta_{\alpha_n}),
\]
which completes the proof.
\end{proof}    

\begin{lemma}\label{l5.2}
    Let $0 < \alpha_1, \alpha_2,\dots,\alpha_n < 2$ and $\prod\limits_{{j=1}}^n\alpha_j < 2$. Let $\alpha_0$ be any complex number and let $\eta$ satisfy the condition \eqref{5.1}.
If
$
z_j \in \Omega^{(+)}({\varepsilon}_{\alpha_j};\eta_{\alpha_j})
$,
where
$
{\varepsilon}_{\alpha_j}={\varepsilon}^{\frac{1}{\beta_j}}$ and $\eta_{\alpha_j}=\cfrac{\eta}{\beta_j}$, $j=1,\dots,n,$
then the following Hankel integral representation holds:
$$\begin{aligned}
E_{\alpha_1, \dots, \alpha_n }(z_1, \dots, z_n;\mu)
&=
\sum_{j=1}^{n}
\frac{1}{\alpha_j}\,
\frac{
e^{z_j^{1/\alpha_j}}
\,z_j^{\frac{\gamma_j+1-\mu}{\alpha_j}}
}{
\prod\limits_{\substack{i=1\\i\neq j}}^{n}
\left(z_j^{\alpha_i/\alpha_j}-z_i\right)
}\\&+
\frac{1}{2\pi i}\,
\frac{1}{\prod\limits_{{j=1}}^n\alpha_j}
\int_{\omega(\varepsilon;\eta)}
\cfrac{
e^{s^{1/\prod\limits_{{j=1}}^n\alpha_j}}
\,s^{\frac{\sum\limits_{j=1}^n\alpha_j+1-\mu}{\prod\limits_{j=1}^n\alpha_j}-1}
}
{
\prod\limits_{j=1}^n(s^{1/\beta_j}-z_j)
}
\,ds.
\end{aligned}$$

\end{lemma}

\begin{proof}
    First, suppose that only the point $z_n$ lies to the right of the Hankel contour
$\omega(\varepsilon_{\alpha_n};\eta_{\alpha_n})$, namely,
$
z_n\in\Omega^{(+)}(\varepsilon_{\alpha_n};\eta_{\alpha_n}).
$
Choose an arbitrary number $\widetilde{\varepsilon}_{\alpha_n}$ satisfying
$\widetilde{\varepsilon}_{\alpha_n}>|z_n|$. Then
$
z_n\in\Omega^{(-)}(\widetilde{\varepsilon}_{\alpha_n};\eta_{\alpha_n}).
$
Furthermore, setting
$
\widetilde{\varepsilon}_{\alpha_1}
=\widetilde{\varepsilon}_{\alpha_n}^{\frac{\alpha_1}{\alpha_n}},
$\dots$,
\widetilde{\varepsilon}_{\alpha_{n-1}}
=\widetilde{\varepsilon}_{\alpha_n}^{\frac{\alpha_{n-1}}{\alpha_n}},
$
it follows that
$
z_j\in\Omega^{(-)}(\widetilde{\varepsilon}_{\alpha_j};\eta_{\alpha_j}),
\,\,
j=1,\dots,n-1.
$
Hence, applying the integral representation \eqref{5.6}, we arrive at
\begin{equation}\label{5.7}E_{\alpha_1,\dots,\alpha_n}(z_1,\dots,z_n;\mu)=
\frac{1}{2\pi i}\,
\frac{1}{\alpha_n}
\int_{\omega(\widetilde{\varepsilon}_{\alpha_n};\eta_{\alpha_n})}
\frac{
e^{s^{1/\alpha_n}}
\,s^{\frac{\gamma_n+1-\mu}{\alpha_n}}
}
{
\left(\prod\limits_{j=1}^{n-1}(s^{\alpha_j/\alpha_n}-z_j)\right)(s-z_n)
}
\,ds.\end{equation}

Moreover, if
$
\varepsilon_{\alpha_n}<|z_n|<\widetilde{\varepsilon}_{\alpha_n},
$
then
$
|\arg z_n|<\eta_{\alpha_n},
$
and, by Cauchy's theorem,
\begin{equation}\label{5.8}\begin{aligned}
E_{\alpha_1,\dots,\alpha_n}(z_1,\dots,z_n;\mu)&=
\frac{1}{2\pi i}\,
\frac{1}{\alpha_n}
\int_{\omega(\widetilde{\varepsilon}_{\alpha_n};\eta_{\alpha_n})-\omega(\varepsilon_{\alpha_n};\eta_{\alpha_n})}
\frac{
e^{s^{1/\alpha_n}}
\,s^{\frac{\gamma_n+1-\mu}{\alpha_n}}
}
{
\left(\prod\limits_{j=1}^{n-1}(s^{\alpha_j/\alpha_n}-z_j)\right)(s-z_n)
}
\,ds\\
&=\frac{1}{\alpha_n}
\frac{
e^{z_n^{1/\alpha_n}}
\,z_n^{\frac{\gamma_n+1-\mu}{\alpha_n}}}
{\prod\limits_{j=1}^{n-1}(z_n^{\alpha_j/\alpha_n}-z_j)}.\end{aligned}
\end{equation}

Therefore, from \eqref{5.7} and \eqref{5.8}, we obtain the following integral representation 
\begin{equation}\label{5.9}\begin{aligned}
E_{\alpha_1,\dots,\alpha_n}(z_1,\dots,z_n;\mu)
&=\frac{1}{\alpha_n}
\frac{
e^{z_n^{1/\alpha_n}}
\,z_n^{\frac{\gamma_n+1-\mu}{\alpha_n}}}
{\prod\limits_{j=1}^{n-1}(z_n^{\alpha_j/\alpha_n}-z_j)}\\
&+
\frac{1}{2\pi i}\,
\frac{1}{\prod\limits_{{j=1}}^n\alpha_j}
\int_{\omega(\varepsilon;\eta)}
\cfrac{
e^{s^{1/\prod\limits_{{j=1}}^n\alpha_j}}
\,s^{\frac{\sum\limits_{j=1}^n\alpha_j+1-\mu}{\prod\limits_{j=1}^n\alpha_j}-1}
}
{
\prod\limits_{j=1}^n(s^{1/\beta_j}-z_j)
}
\,ds.\end{aligned}\end{equation}

Continuation of the proof of Lemma~\ref{l5.2}. Following the same arguments as in the proof of Lemma~\ref{l3.2}, we successively assume that the points $z_{n-1},\ldots,z_1$ lie to the right of the Hankel contour. At each step, only one additional residue is produced, while the remaining integral preserves the same structure with one fewer variable lying to the right of the Hankel contour. Repeating this argument successively yields the statement of Lemma~\ref{l5.2}.

\end{proof}

\begin{theorem}\label{thm5.3}
    Let $0 < \alpha_1, \alpha_2,\dots,\alpha_n < 2$ and $\prod\limits_{{j=1}}^n\alpha_j < 2$. Let $\alpha_0$ be any complex number and
$\tau$ be any real number satisfying the inequalities
$$
\frac{\pi}{2}\prod_{{j=1}}^n\alpha_j
<
\tau
<
\min\!\Bigl(\pi,\pi\prod_{{j=1}}^n\alpha_j\Bigr).
$$

Then, for all integers $r_j\geq 1$, $j=1,\dots,n,$ the function
$E_{\alpha_1, \dots, \alpha_n }(z_1, \dots, z_n;\mu)$ verifies the following asymptotic formulas
drawn from its integral representations whenever $|z_j|\to\infty$, $j=1,\dots,n$.
\begin{itemize}
    \item If $\tau/\beta_j \leq |\arg z_j| \leq \pi$, where $j=1,\dots,n$, then
$$\begin{aligned}E_{\alpha_1, \dots, \alpha_n }(z_1, \dots, z_n;\mu)
&=(-1)^n\sum_{k_1=1}^{r_{1}}
\sum_{k_2=1}^{r_{2}}\dots
\sum_{k_n=1}^{r_{n}}
\frac{\prod\limits_{j=1}^nz_j^{-k_j}}{\Gamma(\mu-\sum_{j=1}^nk_j\alpha_j)}\\
&+
\sum_{j=1}^{n}
o\!\left(
|z_j|^{-r_j}
\prod_{p=1}^{n}|z_p|^{-1}
\right).\end{aligned}
$$

\item If $ |\arg z_j| \leq \tau/\beta_j $, where $j=1,\dots,n$, then

$$\begin{aligned}
E_{\alpha_1, \dots, \alpha_n }&(z_1, \dots, z_n;\mu)
=
\sum_{j=1}^{n}
\frac{1}{\alpha_j}\,
\frac{
e^{z_j^{1/\alpha_j}}
\,z_j^{\frac{\gamma_j+1-\mu}{\alpha_j}}
}{
\prod\limits_{\substack{i=1\\i\neq j}}^{n}
\left(z_j^{\alpha_i/\alpha_j}-z_i\right)
}\\
&+(-1)^n\sum_{k_1=1}^{r_{1}}
\sum_{k_2=1}^{r_{2}}\dots
\sum_{k_n=1}^{r_{n}}
\frac{\prod\limits_{j=1}^nz_j^{-k_j}}{\Gamma\left(\mu-\sum\limits_{j=1}^nk_j\alpha_j\right)}+\sum_{j=1}^no\!\left(|z_j|^{-r_{j}}\Bigl|\prod_{p=1}^nz_p\Bigr|^{-1}\right).\end{aligned}$$
\end{itemize}
\end{theorem}
\begin{proof}
The denominator appearing in the integral representation of the multivariate Mittag--Leffler function (see Lemma \ref{l5.1}) can be decomposed as follows:
\begin{equation}\label{5.10}
\frac{
1}
{
\prod\limits_{j=1}^n(s^{1/\beta_j}-z_j)
}=\frac{
1}
{
\prod\limits_{j=1}^n(A_j-z_j)
}.
\end{equation}
For all integers $r\geq 1$, the following identity holds:
\begin{equation}\label{5.11}
\frac{1}{A-z}=-\sum_{k=1}^r\frac{A^{k-1}}{z^k}+\frac{A^r}{z^r(A-z)}.
\end{equation}
Therefore, \eqref{5.10} can be rewritten in the following form:
\begin{equation}\label{5.12}
\frac{
1}
{
\prod\limits_{j=1}^n(A_j-z_j)
}=\prod_{j=1}^n\left(-\sum_{k_j=1}^{r_j}\frac{A_j^{k_j-1}}{z_j^{k_j}}+\frac{A_j^{r_j}}{z_j^{r_j}(A_j-z_j)}\right).
\end{equation}
For convenience, let
\[
B_j:=
-\sum_{k_j=1}^{r_j}
\frac{A_j^{k_j-1}}{z_j^{k_j}},
\qquad
C_j:=
\frac{A_j^{r_j}}
{z_j^{r_j}(A_j-z_j)},
\quad j=1,\ldots,n.
\]
Then we obtain
\begin{equation}\label{5.13}
\prod_{j=1}^{n}(B_j+C_j)
=
\sum_{I\subseteq[n]}
\left(
\prod_{j\in I}C_j
\right)
\left(
\prod_{j\notin I}B_j
\right).
\end{equation}
Here, $[n]:=\{1,2,\ldots,n\}$, and the summation is taken over all subsets
$I\subseteq[n]$. For each subset $I$, the factors $C_j$ are chosen for
$j\in I$, while the factors $B_j$ are chosen for $j\notin I$.

For example, when $n=3$, we have
$
[3]=\{1,2,3\},
$
whose subsets are
\[
\begin{array}{c|c}
I & \displaystyle\prod_{j\in I}C_j\prod_{j\notin I}B_j\\
\hline
\varnothing & B_1B_2B_3\\
\{1\} & C_1B_2B_3\\
\{2\} & B_1C_2B_3\\
\{3\} & B_1B_2C_3\\
\{1,2\} & C_1C_2B_3\\
\{1,3\} & C_1B_2C_3\\
\{2,3\} & B_1C_2C_3\\
\{1,2,3\} & C_1C_2C_3.
\end{array}
\]
Therefore,
\begin{equation}\label{5.14}
\begin{aligned}
(B_1+C_1)(B_2+C_2)(B_3+C_3)
&=
B_1B_2B_3
+C_1B_2B_3
+B_1C_2B_3
+B_1B_2C_3\\
&
+C_1C_2B_3
+C_1B_2C_3
+B_1C_2C_3
+C_1C_2C_3.
\end{aligned}
\end{equation}
Using \eqref{5.11}, all factors $B_j$ appearing outside the first summation are rewritten as
\begin{equation}\label{5.15}
B_j=
-\sum_{k_j=1}^{r_j}
\frac{A_j^{k_j-1}}{z_j^{k_j}}
=
\frac{z_j^{r_j}-A_j^{r_j}}
{z_j^{r_j}(A_j-z_j)},\quad j=1,\dots,n.
\end{equation}
Combining \eqref{5.12}, \eqref{5.13} (or equivalently, \eqref{5.14}) and \eqref{5.15}, followed by straightforward algebraic manipulations, yields the decomposition \eqref{3 variable denominator}.

Let us
choose a real number $\theta$ satisfying
\begin{equation}\label{5.16}
\frac{\pi}{2}\prod_{{j=1}}^n\alpha_j
<\theta<
\tau
<
\min\!\Bigl(\pi,\pi\prod_{{j=1}}^n\alpha_j\Bigr).
\end{equation}

Since the integral representation of Lemma \ref{l5.1} is valid for every $\varepsilon>0$, we may choose $\varepsilon=1$ without loss of generality. Substituting the \eqref{5.13} into \eqref{multi variable integral representation}, we obtain
  $$E_{\alpha_1, \dots, \alpha_n }(z_1, \dots, z_n;\mu)$$
  $$=(-1)^n\sum_{k_1=1}^{r_{1}}
\sum_{k_2=1}^{r_{2}}\dots
\sum_{k_n=1}^{r_{n}}
\prod_{j=1}^nz_j^{-k_j}
\left(\int_{\omega(1;\theta)}
e^{s^{1/\prod\limits_{{j=1}}^n\alpha_j}}
\,s^{\frac{\sum\limits_{j=1}^n\alpha_j+1-\mu}{\prod\limits_{j=1}^n\alpha_j}-1}s^{\sum\limits_{j=1}^n\frac{k_j-1}{\beta_j}}ds\right)$$
$$
+\frac{1}{2\pi i}\,
\frac{1}{\prod\limits_{{j=1}}^n\alpha_j}\int_{\omega(1;\theta)}
e^{s^{1/\prod\limits_{{j=1}}^n\alpha_j}}
\,s^{\frac{\sum\limits_{j=1}^n\alpha_j+1-\mu}{\prod\limits_{j=1}^n\alpha_j}-1}$$$$\times\sum_{\varnothing\neq I\subseteq[n]}
\left(
\prod_{j\in I}\frac{s^{\frac{r_j}{\beta_j}}}
{z_j^{r_j}\left(s^{\frac{1}{\beta_j}}-z_j\right)}
\right)
\left(
\prod_{j\notin I}\frac{z_j^{r_j}-s^{\frac{r_j}{\beta_j}}}
{z_j^{r_j}\left(s^{\frac{1}{\beta_j}}-z_j\right)}
\right)ds.
$$
By simplifying the first expression, we obtain
$$(-1)^n\sum_{k_1=1}^{r_{1}}
\sum_{k_2=1}^{r_{2}}\dots
\sum_{k_n=1}^{r_{n}}
\prod_{j=1}^nz_j^{-k_j}
\left(\int_{\omega(1;\theta)}
e^{s^{1/\prod\limits_{{j=1}}^n\alpha_j}}
\,s^{\frac{1-\left(\mu-\sum\limits_{j=1}^nk_j\alpha_j\right)}{\prod\limits_{j=1}^n\alpha_j}-1}ds\right).$$
Using the \eqref{gamma integral representation}, we get
$$E_{\alpha_1, \dots, \alpha_n }(z_1, \dots, z_n;\mu)
=(-1)^n\sum_{k_1=1}^{r_{1}}
\sum_{k_2=1}^{r_{2}}\dots
\sum_{k_n=1}^{r_{n}}
\frac{\prod\limits_{j=1}^nz_j^{-k_j}}{\Gamma\left(\mu-\sum\limits_{j=1}^nk_j\alpha_j\right)}+M.$$
where
$$M=\frac{1}{2\pi i}\,
\frac{1}{\prod\limits_{{j=1}}^n\alpha_j}\int_{\omega(1;\theta)}
e^{s^{1/\prod\limits_{{j=1}}^n\alpha_j}}
\,s^{\frac{\sum\limits_{j=1}^n\alpha_j+1-\mu}{\prod\limits_{j=1}^n\alpha_j}-1}$$
$$\times\sum_{\varnothing\neq I\subseteq[n]}
\left(
\prod_{j\in I}\frac{s^{\frac{r_j}{\beta_j}}}
{z_j^{r_j}\left(s^{\frac{1}{\beta_j}}-z_j\right)}
\right)
\left(
\prod_{j\notin I}\frac{z_j^{r_j}-s^{\frac{r_j}{\beta_j}}}
{z_j^{r_j}\left(s^{\frac{1}{\beta_j}}-z_j\right)}
\right)ds.$$
The summation inside the integral can be rewritten in the following form:
$$\sum_{\varnothing\neq I\subseteq[n]}
\left(
\prod_{j\in I}\frac{s^{\frac{r_j}{\beta_j}}}
{z_j^{r_j}\left(s^{\frac{1}{\beta_j}}-z_j\right)}
\right)
\left(
\prod_{j\notin I}\frac{z_j^{r_j}-s^{\frac{r_j}{\beta_j}}}
{z_j^{r_j}\left(s^{\frac{1}{\beta_j}}-z_j\right)}
\right)=$$
\begin{equation}\label{5.18}
    \frac{s^{\frac{r_1}{\beta_1}}\prod\limits_{j=2}^n\left(z_j^{r_j}-s^{\frac{r_j}{\beta_j}}\right)+\dots+s^{\frac{r_i}{\beta_i}}\prod\limits_{\substack{j=1\\i\neq j}}^{n}\left(z_j^{r_j}-s^{\frac{r_j}{\beta_j}}\right)+\dots+s^{\frac{r_n}{\beta_n}}\prod\limits_{j=1}^{n-1}\left(z_j^{r_j}-s^{\frac{r_j}{\beta_j}}\right)}{\prod\limits_{j=1}^nz_j^{r_j}\left(s^{\frac{1}{\beta_j}}-z_j\right)}\end{equation}
\begin{equation}\label{5.19}
    +\frac{\sum\limits_{K\subseteq[n]}
\left(
\prod\limits_{j\in K}s^{\frac{r_j}{\beta_j}}
\right)
\left(
\prod\limits_{j\notin K}\left(z_j^{r_j}-s^{\frac{r_j}{\beta_j}}\right)
\right)}{\prod\limits_{j=1}^nz_j^{r_j}\left(s^{\frac{1}{\beta_j}}-z_j\right)},\end{equation}
where $K=I/\left\{\{\varnothing\},\{1\},\{2\},\dots,\{n\}\right\}$. 

Substituting \eqref{5.18} and \eqref{5.19} into the integral, we obtain

\begin{equation}\label{5.20}M=\frac{1}{2\pi i}\,
\frac{1}{\prod\limits_{{j=1}}^n\alpha_j}\int_{\omega(1;\theta)}
e^{s^{1/\prod\limits_{{j=1}}^n\alpha_j}}
\,s^{\frac{\sum\limits_{j=1}^n\alpha_j+1-\mu}{\prod\limits_{j=1}^n\alpha_j}-1}\end{equation}
$$\times
\left(\frac{s^{\frac{r_1}{\beta_1}}\prod\limits_{j=2}^n\left(z_j^{r_j}-s^{\frac{r_j}{\beta_j}}\right)+\dots+s^{\frac{r_i}{\beta_i}}\prod\limits_{\substack{j=1\\i\neq j}}^{n}\left(z_j^{r_j}-s^{\frac{r_j}{\beta_j}}\right)+\dots+s^{\frac{r_n}{\beta_n}}\prod\limits_{j=1}^{n-1}\left(z_j^{r_j}-s^{\frac{r_j}{\beta_j}}\right)}{\prod\limits_{j=1}^nz_j^{r_j}\left(s^{\frac{1}{\beta_j}}-z_j\right)}\right.$$
$$\left.+\frac{\sum\limits_{K\subseteq[n]}
\left(
\prod\limits_{j\in K}s^{\frac{r_j}{\beta_j}}
\right)
\left(
\prod\limits_{j\notin K}\left(z_j^{r_j}-s^{\frac{r_j}{\beta_j}}\right)
\right)}{\prod\limits_{j=1}^nz_j^{r_j}\left(s^{\frac{1}{\beta_j}}-z_j\right)}\right).$$

Provided that
$
\cfrac{\tau}{\beta_j}\leq|\arg z_j|\le\pi,
$
for sufficiently large $|z_j|$, one has
\begin{equation}\label{5.21}
\min_{s\in\omega(1;\theta)}
\left|s^{1/\beta_j}-z_j\right|
= |z_j|
\sin\left(\cfrac{\tau}{\beta_j}-\cfrac{\theta}{\beta_j}\right)=|z_j|\sin\left(\cfrac{\tau-\theta}{\beta_j}\right),\,\,\, j=1,\dots,n.
\end{equation}

Using \eqref{5.21} together with the following inequality
$$\left|\sum\limits_{m=1}^ns^{\frac{r_m}{\beta_m}}
\prod\limits_{\substack{j=1\\j\neq m}}^n
\left(
z_j^{r_j}-s^{\frac{r_j}{\beta_j}}
\right)
\right|
\le
\sum\limits_{m=1}^n\left|s\right|^{\frac{r_m}{\beta_m}}\prod\limits_{\substack{j=1\\j\neq m}}^n
\left(
|z_j|^{r_j}
+
|s|^{\frac{r_j}{\beta_j}}
\right)$$
we get 
$$|M|\le
\frac{C}{2\pi \prod\limits_{{j=1}}^n\alpha_j}\left[\sum_{j=1}^n|z_j|^{-r_j}\Bigg|\prod_{p=1}^nz_p\Bigg|^{-1}\int_{\omega(1;\theta)}
\left|e^{s^{1/\prod\limits_{{j=1}}^n\alpha_j}}
\right|\left|s^{\frac{\sum\limits_{j=1}^n\alpha_j+1-\mu}{\prod\limits_{j=1}^n\alpha_j}-1}\right|\left|s\right|^{\frac{r_j}{\beta_j}}\left|ds\right|\right.$$
$$
+\Bigg|\prod_{p=1}^nz_p\Bigg|^{-1}\int_{\omega(1;\theta)}
\left|e^{s^{1/\prod\limits_{{j=1}}^n\alpha_j}}
\right|\left|s^{\frac{\sum\limits_{j=1}^n\alpha_j+1-\mu}{\prod\limits_{j=1}^n\alpha_j}-1}\right|\left|\frac{\sum\limits_{m=1}^n\left|s\right|^{\frac{r_m}{\beta_m}}\prod\limits_{\substack{j=1\\j\neq m}}^n\left|s\right|^{\frac{r_j}{\beta_j}}}{\prod\limits_{j=1}^nz_j^{r_j}}\right|\left|ds\right|
$$
$$+\left.\Bigg|\prod_{p=1}^nz_p\Bigg|^{-1}\int_{\omega(1;\theta)}\left|e^{s^{1/\prod\limits_{{j=1}}^n\alpha_j}}
\right|\left|s^{\frac{\sum\limits_{j=1}^n\alpha_j+1-\mu}{\prod\limits_{j=1}^n\alpha_j}-1}\right|\left|\frac{\sum \limits_{K\subseteq[n]}
\left(
\prod\limits_{j\in K}s^{\frac{r_j}{\beta_j}}
\right)
\left(
\prod\limits_{j\notin K}\left(z_j^{r_j}-s^{\frac{r_j}{\beta_j}}\right)
\right)}{\prod\limits_{j=1}^nz_j^{r_j}}\right||ds|\right],
$$
where $C>0$ is a constant.

Observe that, in \eqref{5.19}, the set $K$ contains at least two indices. Therefore, in every summand, at least two factors of the form
$
s^{\frac{r_j}{\beta_j}}
$
appear, whereas in the terms of \eqref{5.18} only one such factor is present. Consequently, compared with the terms in \eqref{5.18}, at least one factor
$
z_j^{r_j}-s^{\frac{r_j}{\beta_j}}
$
is replaced by the corresponding power
$
s^{\frac{r_j}{\beta_j}}.
$
After division by
$
\prod\limits_{j=1}^{n}z_j^{r_j},
$
every such replacement produces one additional factor of the form
$
z_j^{-r_j}.
$
Hence, each summand in \eqref{5.19} contains at least one extra negative power of one of the variables $z_j$ compared with the corresponding terms in \eqref{5.18}. Therefore, every term in \eqref{5.19} is of strictly smaller order than
\[
\sum_{j=1}^n|z_j|^{-r_j}
\prod_{p=1}^{n}|z_p|^{-1},
\qquad
|z_j|\to\infty,\quad j=1,\ldots,n.
\]
Moreover, the rays defined by
$
S_{\pm\theta}
=
\left\{
\arg s = \pm \theta,\; |s| \ge 1
\right\}
$
belong to the contour $\omega(1;\theta)$, the integral in \eqref{5.20} is convergent; since
\[
\left|e^{s^{1/\prod\limits_{{j=1}}^n\alpha_j}}\right|
=
\exp\!\left(
|s|^{{1/\prod\limits_{{j=1}}^n\alpha_j}}
\cos\!\left(\frac{\theta}{\prod\limits_{{j=1}}^n\alpha_j}\right)
\right).
\]

Now, according to inequalities \eqref{5.16}, we have that $\cos\!\left(\cfrac{\theta}{\prod\limits_{{j=1}}^n\alpha_j}\right)<0$. Thus,
\[
M=
\sum_{j=1}^{n}
o\!\left(
|z_j|^{-r_j}
\prod_{p=1}^{n}|z_p|^{-1}
\right),\qquad |z_j|\to\infty,\qquad j=1,\dots,n.
\]

If $ |\arg z_j| \leq \tau/\beta_j$, using the corresponding integral representations established in Lemma \ref{l5.2}, the proof proceeds in the same way as the above one.

Thus Theorem \ref{thm5.3} is proved. \qed
\end{proof}

\section{Conclusion}

In this paper, a multivariable Mittag--Leffler-type function arising in the theory of fractional differential equations with several fractional parameters has been investigated. New Hankel contour integral representations have been established for the three-variable Mittag--Leffler function. Based on these representations, complete asymptotic expansions have been derived in different sectors of the complex plane.
Furthermore, the obtained results have been extended to the multivariable Mittag--Leffler function of an arbitrary number of variables. The proposed approach provides a unified framework for deriving integral representations and asymptotic formulas for multivariable Mittag--Leffler functions. The obtained asymptotic expansions generalize several previously known results for one- and two-variable Mittag--Leffler functions.
The results presented in this paper provide useful analytical tools for the qualitative analysis of fractional differential equations involving several fractional derivatives. In particular, they may be applied to the investigation of asymptotic behavior, decay estimates, explicit solution representations, and related problems arising in fractional evolution equations. It is expected that the proposed method can also be adapted to other classes of generalized special functions appearing in fractional calculus.

\section*{Acknowledgements}

The author is grateful to professor R. Ashurov for discussions of these results.

\medskip

\begin{tabular}{p{9cm}}
Damir Shamuratov\\
V.I. Romanovskiy Institute of Mathematics Uzbekistan Academy of Sciences, University street 9, Tashkent--100174, Uzbekistan\\
email: damirshamuratov5@gmail.com\\
https://orcid.org/0009-0004-1566-7171
\end{tabular}


\begin{thebibliography}{20}

\bibitem{Ashurov}
{\sc R. Ashurov and D. Shamuratov},
{\it Inverse Problem for a Multi-Term Time-Fractional Diffusion Equation with the Caputo Derivatives},
arXiv:2603.01833, (2026).
https://doi.org/10.48550/arXiv.2603.01833.

\bibitem{Gorenflo}
{\sc R. Gorenflo, A.~A. Kilbas, F. Mainardi and S. Rogosin},
{\it Mittag--Leffler Functions, Related Topics and Applications},
2nd ed.,
Springer Monographs in Mathematics,
Springer, Berlin, 2014.

\bibitem{Kilbas}
{\sc A.~A. Kilbas, H.~M. Srivastava and J.~J. Trujillo},
{\it Theory and Applications of Fractional Differential Equations},
North-Holland Mathematics Studies, Vol.~204,
Elsevier, Amsterdam, 2006.

\bibitem{Kiryakova}
{\sc V. Kiryakova},
{\it The Multi-Index Mittag--Leffler Functions as Important Special Functions of Fractional Calculus},
Comput. Math. Appl.
{\bf 59} (2010), No.~5, 1885--1895.

\bibitem{Lavault}
{\sc C. Lavault},
{\it Integral Representations and Asymptotic Behaviour of a Mittag--Leffler Type Function of Two Variables},
Adv. Oper. Theory
{\bf 3} (2018), No.~2, 40--48.

\bibitem{Yamamoto}
{\sc Z. Li, Y. Liu and M. Yamamoto},
{\it Initial-Boundary Value Problems for Multi-Term Time-Fractional Diffusion Equations with Positive Constant Coefficients},
Appl. Math. Comput.
{\bf 257} (2015), 381--397.

\bibitem{Luchko}
{\sc Yu. Luchko},
{\it Operational Method in Fractional Calculus},
Fract. Calc. Appl. Anal.
{\bf 2} (1999), No.~4, 463--489.

\bibitem{Machado}
{\sc J.~A. Tenreiro Machado and A.~M. Lopes},
{\it Fractional Van der Pol Oscillator},
In:
{\it Handbook of Fractional Calculus with Applications},
Vol.~4,
De Gruyter, Berlin, 2019, 2--21.

\bibitem{Miller}
{\sc K.~S. Miller and B. Ross},
{\it An Introduction to the Fractional Calculus and Fractional Differential Equations},
John Wiley \& Sons,
New York, 1993.

\bibitem{Ogorodnikov1}
{\sc E.~N. Ogorodnikov},
{\it Mathematical Models of the Fractional Oscillator, Setting and Structure of the Cauchy Problem},
Proc. Sixth All-Russian Sci. Conf. with Int. Participation,
Part~1,
Samara State Technical Univ.,
Samara, 2009, 177--181.

\bibitem{Ogorodnikov2}
{\sc E.~N. Ogorodnikov and N.~S. Yashagin},
{\it Setting and Solving of the Cauchy Type Problems for the Second-Order Differential Equations with Riemann--Liouville Fractional Derivatives},
Vestn. Samar. Gos. Tekhn. Univ. Ser. Fiz.-Mat. Nauki
{\bf 20} (2010), No.~1, 24--36.

\bibitem{Podlubny}
{\sc I. Podlubny},
{\it Fractional Differential Equations},
Math. Sci. Eng.,
Vol.~198,
Academic Press,
San Diego, 1999.

\bibitem{Umarov}
{\sc S. Umarov},
{\it The Fractional Duhamel Principle for Systems of Fractional Multi-Term Differential-Operator Equations},
Fract. Calc. Appl. Anal.
(2026).
https://doi.org/10.1007/s13540-026-00554-1.

\end{thebibliography}
\end{document}